\documentclass[12pt,a4paper,reqno]{amsart}
\usepackage[margin=28mm]{geometry}
\usepackage{amsmath,amssymb,amsthm,mathtools,booktabs}
\usepackage{xcolor,array,longtable,placeins}
\newcommand{\note}[2][blue]{\textcolor{#1}{#2}}
\newcommand{\rev}[1]{\note[blue]{#1}}

\usepackage[T1]{fontenc}
\usepackage{lmodern}
\usepackage{microtype}
\usepackage[colorlinks=true,linkcolor=blue,citecolor=blue,urlcolor=blue]{hyperref}
\hypersetup{pdftitle={Rank-Two Drinfeld Modules over Elliptic Coordinate Rings},pdfauthor={Chuangqiang Hu; Ze-Dong Huang; Chang-An Zhao}}
\allowdisplaybreaks
\makeatletter
\renewcommand\subsection{\@startsection{subsection}{2}{\z@}%
  {-2.5ex plus -1ex minus -.2ex}{1.25ex plus .2ex}%
  {\normalfont\bfseries}}
\makeatother

\newtheorem{thm}{Theorem}[section]
\newtheorem{prop}[thm]{Proposition}
\newtheorem{lem}[thm]{Lemma}
\newtheorem{cor}[thm]{Corollary}
\newtheorem{thmletter}{Theorem}

\theoremstyle{definition}

\newtheorem{remark}[thm]{Remark}
\newtheorem{example}[thm]{Example}
\newcommand{\A}{\mathbf A}
\newcommand{\Fq}{\mathbb F_q}
\newcommand{\KK}{\mathbb K}
\newcommand{\s}{\mathbf s}
\newcommand{\ord}{\operatorname{ord}}
\newcommand{\Sgn}{\operatorname{Sgn}}
\newcommand{\Fr}{\operatorname{Fr}}
\newcommand{\Spec}{\operatorname{Spec}}
\newcommand{\Supp}{\operatorname{Supp}}
\newcommand{\Span}{\operatorname{Span}}
\newcommand{\uk}{\mathbf u}
\newcommand{\vk}{\mathbf v}

\newcommand{\cN}{\mathcal N}
\newcommand{\cH}{\mathcal H}
\newcommand{\cJ}{\mathcal J}
\title[Rank-Two Drinfeld Modules]{Rank-Two Drinfeld Modules over Elliptic Coordinate Rings}
\author{Chuangqiang Hu}
\address{School of Mathematics, Sun Yat-Sen University, Guangzhou, China}
\email{huchq5@mail.sysu.edu.cn}
\author{Ze-Dong Huang}
\address{School of Mathematics, Sun Yat-Sen University, Guangzhou, China}
\email{huangzd6@mail2.sysu.edu.cn}
\author{Chang-An Zhao}
\address{School of Mathematics, Sun Yat-Sen University, Guangzhou, China}
\email{zhaochan3@mail.sysu.edu.cn}
\thanks{Corresponding author: Chang-An Zhao.}
\date{}
\begin{document}
\begin{abstract}
Let ${\mathbf A}$ be the affine coordinate ring of $E\setminus\{\infty\}$,
where $E$ is an elliptic curve over $\mathbb F_q$ and $\infty$ is rational. We solve the rank-two case of the open problem of
explicitly classifying sign-normalized Drinfeld ${\mathbf A}$-modules in generic characteristic with a fixed determinant motive; the classification for rank greater than two remains open. Outside an explicit exceptional locus, the rank-two Drinfeld modules are parametrized by the open modular curve $Y^{q+1}=\pi(X)$ with $Y\neq0$. The polynomial $\pi$ is squarefree, and the smooth projective completion of this curve has genus $q^2$. The roots of $\pi$ split into two types. Above roots of the first type, all coefficient functions extend regularly, and their extensions yield exactly the exceptional modules.  Roots of the second type correspond to cusps of the modular curve; above these roots,  the coefficient family has poles.
\end{abstract}
\maketitle
 
\section{Introduction}
\subsection{Relation to earlier work}
Drinfeld introduced Drinfeld modules in \cite{Drinfeld74}. Anderson's motive
formalism \cite{Anderson86} and the determinant construction of van der
Heiden \cite{vdHGJ04} connect them to linear algebra over the coordinate
ring. For rank one, the sign-normalized theory of Hayes \cite{Hayes79}
and the shtuka functions of Thakur \cite{Thakur93} give effective tools.
Explicit constructions on elliptic curves appear in
Dummit--Hayes \cite{DummitHayes94}; Green--Papanikolas
\cite{GreenPapanikolas18} develop the associated motives and explicit
shtuka formulas in detail. We use this rank-one theory for the determinant
of a rank-two motive. General background is given in
\cite{Goss96,Thakur04,Papikian23}.

Rank-two modules over elliptic coordinate rings also occur in the
analytic theory of Drinfeld modular curves. Keller
\cite[Definition~4.5.17 and \S5.5]{Keller05} constructs $j$-invariants
from theta functions, expresses their $(q-1)$st powers as quotients of
modular forms when the ideal class group is nontrivial, and computes
expansions at a cusp. The genus-zero property of the modular curves
for the full modular groups in this setting is already part of that
theory; see also the introduction of \cite{Keller05}.
Our parameter $\kappa$ is recovered directly from the coefficients of
the motive relation. We do not identify it with a specified analytic
$j$-invariant of Keller.

Bassa--Beelen--Nguyen give a direct explicit construction over an
elliptic coordinate ring \cite[\S4.1]{BassaBeelenNguyen14}.
They compute normalized rank-two modules over
$\mathbb F_2[S,T]/(S^2+S-T^3-T)$ in finite characteristic $(S,T)$,
using Gr\"obner bases. Their calculation gives coefficient curves of
genus four and invariant curves of genus zero.
The present construction works uniformly for smooth elliptic curves
over $\Fq$ in generic characteristic, after fixing the determinant.
It gives the equation $Y^{q+1}=\pi(X)$ together with a classification
of its roots and a criterion for regular extension of every coefficient.

For a different non-polynomial ring, obtained by removing a point of
degree two from the projective line, Hu--Huang \cite{HuHuang24}
give an explicit family of rank-two modules using Anderson motives.
The genus-one coordinate ring considered here requires a degree-two
divisor and a separate analysis of the locus $P_3=0$.

Theorems~\ref{thm:main}--\ref{thm:boundary} extend these explicit
constructions by treating every smooth elliptic curve over $\Fq$ in
generic characteristic and by including the boundary locus $P_3=0$ in
the same coefficient family. The number of modules with the larger scalar automorphism group agrees
with the classical elliptic-point count; see
\cite[Corollary~2.6]{MasonSchweizer15}. Our construction identifies these
modules by the equation $Q\boxplus Q^{(1)}=V$ and by the vanishing of
$P_3$. 
finite points with $Y=0$.

\subsection{Aim and setting}
The aim of this paper is to determine all sign-normalized rank-two
Drinfeld modules over an elliptic coordinate ring in generic
characteristic with a fixed determinant motive, and to describe the
exceptional locus of the resulting coefficient formulas.

Let
\[
 E:\quad y^2+a_1xy+a_3y=x^3+a_2x^2+a_4x+a_6,\qquad a_i\in\Fq,
\]
be a smooth curve, with identity $\infty$. Write
$\A=\Fq[x,y]$ for its affine coordinate ring, with the displayed relation
understood. Let $\KK=\Fq(\theta,\eta)$ be a second copy of its function
field, let $L\supset\KK$ be algebraically closed, and put
$\A_L=L\otimes_{\Fq}\A$ and $\xi=(\theta,\eta)$.
Thus the characteristic is generic throughout.
Fix a point $V=(\alpha,\beta)$ such that
\[
 V\boxplus(V^{(1)})^\vee=\xi.
\]
Here $\boxplus$ is the group law, $P^\vee$ is the inverse point, and
$P^{(1)}$ is its $q$-power Frobenius image.
Put $\beta^\vee=-\beta-a_1\alpha-a_3$ and define
$m=(\beta^q-\beta^\vee)/(\alpha^q-\alpha)$.
The denominator is nonzero. 
Let $\mathcal L_m=y-\eta-m(x-\theta)$. 
The line $\mathcal L_m=0$ passes through $V^\vee$ and $V^{(1)}$, with $\xi$ as its third intersection counted with multiplicity.
And its divisor is:
\begin{equation}\label{eq:line-divisor}
 (\mathcal L_m) =V^\vee+V^{(1)}+\xi-3\infty.
\end{equation} 

A rank-two Drinfeld module is sign-normalized if the leading coefficient
of $\phi_a$ is $\Sgn(a)$ for every nonzero $a\in\A$, where the sign is
specified in Section~\ref{sec:Pre}. In particular it has the form
\begin{align}
 \phi_x&=\theta+g_1\tau+g_2\tau^2+g_3\tau^3+\tau^4,\label{eq:phix2}\\
 \phi_y&=\eta+h_1\tau+h_2\tau^2+h_3\tau^3+h_4\tau^4+h_5\tau^5+\tau^6.
 \label{eq:phiy2}
\end{align}
Define
\begin{align}
 P_0&=h_4-g_2^{q^2}-(h_5-g_3^{q^2})g_3^q,\notag\\
 P_3&=h_3-g_1^{q^2}-(h_5-g_3^{q^2})g_2^q-P_0g_3.\label{eq:P0P3def}
\end{align}
For a Drinfeld module $\phi$, write $M_\phi=L\{\tau\} \s $ for its associated motive, where $\tau c=c^q\tau$ and $\s$ is the cyclic generator. 

The polynomial that controls the construction is
\begin{align}
 \pi(X)={}&-X^{2q+1}+mX^{2q}+X^{q+2}-(2m+a_1)X^{q+1}\notag\\
 &+(\alpha+2\alpha^q+a_2)X^q+(m+a_1)X^2\notag\\
 &-(2\alpha+\alpha^q+a_2)X+\beta^\vee-(\beta^\vee)^q.
 \label{eq:pi-expanded}
\end{align}
Let $C_\pi$ be the curve $Y^{q+1}=\pi(X)$ and  $C_\pi^\circ$ be its open subset $Y\ne 0$. Its smooth projective completion is denoted by $\bar C_\pi$.  

For $k\in L$, let $D_k$ be the residual intersection divisor in
\[
 \bigl(y-\beta^\vee-k(x-\alpha)\bigr)=V^\vee+D_k-3\infty.
\]
It has degree two and includes intersection multiplicities.

The determinant condition means an isomorphism of $\A_L$-modules with
semilinear Frobenius
\begin{equation}\label{eq:determinantcondition}
      \bigwedge\nolimits^2 M_\phi
 \simeq \bigl(V^{-1}\A_L,\ g\longmapsto f_\alpha g^{(1)}\bigr),
\end{equation}
where $f_\alpha$ is given in \eqref{eq:deffa}. Here $V^{-1}\A_L$ is the
fractional ideal of functions with at most a simple pole at $V$ and no
other finite poles. A choice of this isomorphism is not part of the data.

\subsection{Main theorems}
\begin{thmletter}\label{thm:main}
Let $\mathcal X_V$ be the set of sign-normalized rank-two Drinfeld
$\A$-modules over $L$ with this determinant condition. Its elements are
the homomorphisms $\phi:\A\to L\{\tau\}$ themselves.
\begin{enumerate}
\item For every $\phi\in\mathcal X_V$, the equality
\[
 M_\phi=\A_L\s+\A_L\tau\s+\A_L\tau^2\s
\]
holds if and only if $P_3\ne0$. When $P_3=0$, the quotient by this
three-generator submodule has $L$-dimension one.
\item The modules in $\mathcal X_V$ with $P_3\ne0$ are in bijection with
\[
 C_\pi^\circ(L)=\{(\kappa,\lambda_1):\lambda_1^{q+1}=\pi(\kappa),\ \kappa \in L,
 \lambda_1\ne0\}.
\]
Their coefficients are given in \eqref{eq:allcoefficients}, and
$P_3=-\lambda_1$. For fixed $V$, two such modules are isomorphic if and
only if their values of $\kappa$ agree.
\end{enumerate}
\end{thmletter}


\begin{thmletter}\label{thm:paramcurve}
\begin{enumerate}
\item The $2q+1$ roots of $\pi$ are distinct. They are precisely the
slopes for which either
\[
 D_k=R+(V\boxminus R),\qquad
 R\in E(\Fq)\setminus\{\infty\},                         \tag{I}
\]
or
\[
 D_k=Q+Q^{(1)},\qquad Q\boxplus Q^{(1)}=V.                \tag{II}
\]
The two sets are disjoint. Their sizes are respectively
$\#E(\Fq)-1$ and $2(q+1)-\#E(\Fq)$.
\item The affine curve $C_\pi:Y^{q+1}=\pi(X)$ is smooth and irreducible.
Its smooth projective completion has one point over $X=\infty$ and genus
$q^2$. The projection to the $X$-line is totally ramified at the
$2q+1$ finite roots and at infinity.
\end{enumerate}
\end{thmletter}

\begin{thmletter}\label{thm:boundary}
\begin{enumerate}
\item All coefficient functions \eqref{eq:allcoefficients} extend regularly at the points $(k,0)$ of type~\textup{(II)}, exactly corresponding to the modules in $\mathcal X_V$ with $P_3=0$. 
These modules are precisely the modules for which scalar multiplication by $\mathbb F_{q^2}$ commutes with the $\A$-action, and their scalar automorphism group is $\mathbb F_{q^2}^\times$.
\item At every point $(k,0)$ of type~\textup{(I)}, $g_3$ has pole order
$q^3$. At the unique point $P_\infty$ of $\bar C_\pi$ over $X=\infty$,
it has pole order $q^5$. Consequently, the largest open subset of
$\bar C_\pi$ on which all coefficient functions are regular is
\[
 \bar C_\pi\setminus
 \bigl(\{P_\infty\}\cup
       \{(k,0):k\text{ is a type~\textup{(I)} root}\}\bigr).
\]
Its complement has $\#E(\Fq)$ points.
\end{enumerate}
\end{thmletter}

\subsection{Structure and proof strategy}
Section~\ref{sec:Pre} fixes the divisor and motive conventions and
constructs the degree-two quotient. Section~\ref{sec:motive} derives the
quadratic motive relation, reduces Frobenius stability to two scalar
equations, and proves the converse by a pole filtration.
Section~\ref{sec:Moduli} obtains the coefficient formulas, classifies the
roots of $\pi$, and analyzes the locus $P_3=0$ and the regular extensions.
Section~\ref{sec:examples} gives explicit examples. The appendices record
the line-intersection identities and the twisted-polynomial elimination.

\section{Preliminaries}\label{sec:Pre}
\subsection{Divisors, Frobenius, and rank-one motives}
Write $\mathbf K_L=\operatorname{Frac}(\A_L)$. For a divisor $D$ on $E_L=E\times_{\Spec\Fq}\Spec L$, put
\[
 \mathcal L(D)=\{g\in\mathbf K_L:(g)+D \geqslant 0\}\cup\{0\},\qquad
 D^{-1}\A_L=\bigcup_{n \geqslant 0}\mathcal L(D+n\infty).
\]
We use $+$ for divisors and $\boxplus$ for the group law. Frobenius twist
raises coefficients in $L$ to the $q$-th power and fixes $x,y$.
It is denoted by $(-)^{(1)}$ on functions, points, divisors, and fractional ideals.
The twisted polynomial ring $L\{\tau\}$ is defined by $\tau c=c^q\tau$.

Since $\ord_\infty(x)=-2$ and $\ord_\infty(y)=-3$, the monomials
$x^i$ and $x^jy$ have distinct pole orders
$H=\{0,2,3,4,\ldots\}$.
Their leading coefficient defines the multiplicative sign
$\Sgn^L:\mathbf K_L^\times\to L^\times$.
We write $\Sgn=\Sgn^L|_{\A\setminus\{0\}}$, with values in
$\Fq^\times$.
The only units of $\A_L$ are $L^\times$.

The point $V$ is nonconstant: $V\notin E(\overline{\Fq})$, since
$V-V^{(1)}=\xi$ is generic. More generally, no nonzero isogeny defined
over a finite field sends $V$ to a constant point. Its fiber over such
a point is a finite scheme over $\overline{\Fq}$.
In particular $V\ne\infty$, $V\ne V^\vee$, and
$\alpha\ne\alpha^q$.

Substitution of the line $\mathcal L_m$ into the Weierstrass equation $E$ gives a monic cubic with roots $\alpha,\alpha^q,\theta$, counted with multiplicity. 
Its coefficient of $x^2$ and the line evaluations give
\begin{align}
 \alpha+\alpha^q+\theta&=m^2+a_1m-a_2,\label{eq:m-square}\\
 \beta^\vee&=m(\alpha-\theta)+\eta,\qquad
 \beta^q=m(\alpha^q-\theta)+\eta.\label{eq:linevalues}
\end{align}
The sign-normalized shtuka function is
\begin{equation}\label{eq:deffa}
 f_\alpha=\frac{y-\eta-m(x-\theta)}{x-\alpha},\qquad
 (f_\alpha)=V^{(1)}+\xi-V-\infty.
\end{equation}
These formulas are the rank-one construction of
\cite[\S\S2--3]{GreenPapanikolas18}, based on \cite{Thakur93}.
They apply algebraically to every choice of $V$ satisfying the stated
relation and no choice of a completion or of an analytic lattice is needed here.


The motive associated with the rank-one module $\psi$ is $V^{-1}\A_L$, with $\tau g=f_\alpha g^{(1)}$.
Indeed,
\[
 1,f_\alpha,f_\alpha f_\alpha^{(1)},\ldots,
 f_\alpha\cdots f_\alpha^{(n-1)}
\]
form a basis of $\mathcal L(V+n\infty)$: they have distinct pole
orders $0,1,\ldots,n$ at infinity, and this space has dimension $n+1$.
In particular, $V^{-1}\A_L=\A_L+\A_Lf_\alpha$.
Expanding $x \in \mathcal{L}(V + 2 \infty)$ with respect to the basis $\{1, f_{\alpha}, f_{\alpha}f_{\alpha}^{(1)} \}$ gives
\begin{equation}\label{eq:xfalpha}
 x=\theta+(m+m^q+a_1)f_\alpha+f_\alpha f_\alpha^{(1)}.
\end{equation}
For completeness, if
$\psi_x=\theta+x_1\tau+\tau^2$ and
$\psi_y=\eta+y_1\tau+y_2\tau^2+\tau^3$, then
\begin{equation}\label{eq:rankonecoeff}
 x_1=m+m^q+a_1,\quad
 y_1=\frac{x_1(\eta^q-\eta)}{\theta^q-\theta},\quad
 y_2=\frac{\eta^{q^2}-\eta+x_1y_1^q-y_1x_1^q}{\theta^{q^2}-\theta}.
\end{equation}
See \cite[Theorem~3.2 and Proposition~3.3]{GreenPapanikolas18}.

\subsection{Rank-two motives and a quotient of length two}
For a rank-two module $\phi$, use the following left-module convention
\[
 M_\phi=L\{\tau\}\s,\qquad
 a(P\s)=P\phi_a\s\quad(a\in\A,\ P\in L\{\tau\}).
\]

We verify directly that \(M_\phi\) is a projective \(\A_L\)-module
of rank two, as required for the arguments below.
Since $\deg_\tau\phi_x=4$, reduction by leading $\tau$-degree shows
that $\s,\tau\s,\tau^2\s,\tau^3\s$ form an $L[x]$-basis of
$M_\phi$. The module is also torsion-free over $\A_L$: for a nonzero
$a\in\A_L$, expand $a$ in the monomials $x^i,x^jy$.
Their actions on any nonzero $P\s$ have distinct leading
$\tau$-degrees, so the highest term of $a(P\s)$ cannot cancel.
As $\A_L$ is a Dedekind domain, $M_\phi$ is therefore projective.
The extension $\operatorname{Frac}(\A_L)/L(x)$ has degree two,
so its rank over $\A_L$ is $4/2=2$.

The linearization $\tau:M_\phi^{(1)}\to M_\phi$ is injective:
in the $L$-bases of Frobenius powers it sends 
$1 \otimes \tau^i \s$ to $\tau^{i+1}\s$.
Its cokernel is $L\bar\s$, on which $a\in\A$ acts by
$a(\xi)$, and hence is supported at $\xi$.
These elementary properties agree with the dimension-one motive
formalism; see \cite{Anderson86,Goss96,Papikian23} for background.

For the converse needed here, Proposition~\ref{prop:converse}
constructs the Drinfeld module and checks its rank and normalization
directly from the pole filtration.
The determinant construction \cite{vdHGJ04} identifies
$\bigwedge^2M_\phi$ with a rank-one motive. 
By the determinant condition~\eqref{eq:determinantcondition}, the associated point is our $V$.

Below we normalize \(h\) by \(\Sgn^L(h)=1\) and set
\(S_0=h(\s\wedge\tau\s)\). With this normalization, we will obtain
\(\tau S_0=-f_\alpha S_0\). The sign is compatible with the determinant
condition in Theorem~\ref{thm:main}: rescaling the rational determinant
generator shows that the rank-one motives with Frobenius coefficients
\(f_\alpha\) and \(-f_\alpha\) are isomorphic.

\begin{lem}\label{lem:quotient}
For every sign-normalized rank-two module, put
$N=\A_L\s+\A_L\tau\s$.
\begin{enumerate}
\item We have
\[
 N=\A_L\s\oplus\A_L\tau\s,\qquad
 \dim_L(M_\phi/N)=2.
\]
The classes of $\tau^2\s,\tau^3\s$ are an $L$-basis of the quotient $M_\phi/N$.
\item There is an effective finite divisor $D$ of degree two and a function
$h$, unique after imposing $\Sgn^L(h)=1$, such that
\begin{equation}\label{eq:divh}
 (h)=V+\infty-D,\qquad
 \bigwedge^2 M_\phi=hV^{-1}\A_L(\s\wedge\tau\s)
                  =D^{-1}\A_L(\s\wedge\tau\s).
\end{equation}
In particular $D$ avoids $V$ and $\infty$ and its group sum is $V$.
\end{enumerate}
\end{lem}
\begin{proof}
\emph{(i).} The products of the standard monomials of $\A_L$ with $\s$ and
$\tau\s$ have distinct leading $\tau$-degrees $2H$ and $2H+1$.
Their union consists of all nonnegative integers except $2,3$.
Successive subtraction of leading terms proves both the directness and
the assertion about the quotient.

\emph{(ii).} Put $T=M_\phi/N$. Since $T$ is a finite-dimensional $L$-vector space, it is an $\A_L$-module of finite length, supported at finitely many finite points. 
Define
\[
 D=\sum_P d_P P,\qquad
 d_P=\operatorname{length}_{\mathcal O_{E_L,P}}T_P.
\]
To relate $D$ to the determinant, fix a finite point $P$ and a uniformizer $t_P$. 
Both $M_{\phi,P}$ and $N_P$ are free of rank two over the local ring $\mathcal O_{E_L,P}$. 
By the elementary divisor theorem, there are bases in which the inclusion $N_P\hookrightarrow M_{\phi,P}$ is represented by $\operatorname{diag}(t_P^{e_1},t_P^{e_2})$. 
Hence
\[
 T_P\cong\mathcal O_{E_L,P}/(t_P^{e_1})
 \oplus\mathcal O_{E_L,P}/(t_P^{e_2}),\qquad d_P=e_1+e_2.
\]
Taking exterior powers multiplies the determinant generator by
$t_P^{d_P}$. Since
$\bigwedge^2N=\A_L(\s\wedge\tau\s)$, these local calculations give the
global equality
\[
 \bigwedge\nolimits^2M_\phi
 =D^{-1}\A_L(\s\wedge\tau\s).
\]
Moreover, because $L$ is algebraically closed,
\[
 \deg D=\sum_Pd_P=\dim_LT=2.
\]

By the fixed determinant condition, we can also write
\[
 \bigwedge\nolimits^2M_\phi=V^{-1}\A_LS_0,
 \qquad S_0=h(\s\wedge\tau\s)
\]
for a nonzero rational function $h$. Comparing the two expressions gives
$hV^{-1}\A_L=D^{-1}\A_L$, so the finite part of $(h)$ is $V-D$.
Since $\deg V=1$ and $\deg D=2$, the equality $\deg(h)=0$ forces
$\ord_\infty h=1$. Thus
\[
 (h)=V+\infty-D.
\]
If $V$ occurred in $D$, then $1/h$ would be regular at every finite point
and would have exactly a simple pole at $\infty$. This contradicts
$\mathcal L(\infty)=L$, so $D$ avoids $V$; it avoids $\infty$ by
construction. The Abel--Jacobi criterion applied to $(h)$ now shows that
the group sum of $D$ is $V$. Finally, $h$ is determined by its divisor up
to a scalar in $L^\times$, and the condition $\Sgn^L(h)=1$ uniquely fixes
this scalar.

\end{proof}

\begin{cor}\label{cor:dotVddotVnonrational}
For a Drinfeld module, $\Supp D\cap E(\Fq)=\varnothing$.
This includes the case $D=2P$.
\end{cor}
\begin{proof}[A more detailed proof]
Suppose that $P\in\Supp D\cap E(\Fq)$. Since the coordinates of $P$ are
fixed by the $q$-power Frobenius, its maximal ideal
$\mathfrak m_P\subset\A_L$ is Frobenius-stable. Consequently, $\tau$
induces a semilinear map $\bar\tau$ on
\[
 \overline M=M_\phi/\mathfrak m_PM_\phi,
 \qquad \bar\tau(c\bar m)=c^q\bar\tau(\bar m).
\]
The module $M_\phi$ is projective of rank two, so its fiber
$\overline M$ at $P$ has $\dim_L\overline M=2$.

By Lemma~\ref{lem:quotient}, $P\neq V,\infty$. 
Localizing $\bigwedge^2M_\phi=V^{-1}\A_LS_0$ at $P$ therefore gives
\[
 \left(\bigwedge\nolimits^2M_\phi\right)_P
 =\mathcal O_{E_L,P}S_0,
\]
because $(V^{-1}\A_L)_P=\mathcal O_{E_L,P}$. In particular, $S_0$ is a
regular local generator of the determinant. 
Since $P \in D$ and
$S_0=h(\s\wedge\tau\s)$, we have
\[
 \operatorname{ord}_P(h^{-1}) \geqslant 1,\qquad
 \s\wedge\tau\s=h^{-1}S_0
 \in\mathfrak m_P
 \left(\bigwedge\nolimits^2M_\phi\right)_P.
\]
After reduction modulo $\mathfrak m_P$, this yields
\[
 \bar\s\wedge\bar\tau(\bar\s)=0
 \quad\text{in }\bigwedge\nolimits^2\overline M.
\]

If $\bar\s=0$, then every iterate $\bar\tau^n(\bar\s)$ also vanishes.
If $\bar\s\neq0$, the vanishing wedge product in the two-dimensional
vector space $\overline M$ implies
$\bar\tau(\bar\s)\in L\bar\s$. Semilinearity then gives inductively
$\bar\tau^n(\bar\s)\in L\bar\s$ for every $n \geqslant 1$.
Thus in either case all the iterates of $\bar\s$ span an $L$-subspace of
dimension at most one. 

But $M_\phi=L\{\tau\}\s$ implies
that these iterates span $\overline M$, whose dimension is two. This is a
contradiction. 
\end{proof}

\section{The quadratic relation and Frobenius stability}\label{sec:motive}
\subsection{The relation and its divisor}
In addition to \eqref{eq:P0P3def}, put
\begin{equation}\label{eq:P12}
 P_1=h_1-P_0g_1,\qquad
 P_2=h_2-(h_5-g_3^{q^2})g_1^q-P_0g_2.
\end{equation}
Right division of $\phi_y-\eta$ by $\phi_x-\theta$ gives
\begin{align}
 0={}&P_3\tau^3\s+(x-\theta^{q^2}+P_2)\tau^2\s\notag\\
 &+\bigl(P_1+(h_5-g_3^{q^2})(x-\theta^q)\bigr)\tau\s
 +\bigl(P_0(x-\theta)-(y-\eta)\bigr)\s.
 \label{eq:first}
\end{align}
The derivation of~\eqref{eq:first} and the eliminations used in the arguments below are given in Appendix~\ref{app:TwistedDivision}.

\begin{prop}\label{prop:motive2}
For every sign-normalized rank-two module in generic characteristic:
\begin{enumerate}
\item The equality
\begin{equation}\label{eq:three-generator-criterion}
 M_\phi=\A_L\s+\A_L\tau\s+\A_L\tau^2\s
\end{equation}
holds if and only if $P_3\ne0$. 
More precisely,
\begin{equation}\label{eq:three-generator-defect}
 \dim_L\frac{M_\phi}{\A_L\s+\A_L\tau\s+\A_L\tau^2\s}
 =\begin{cases}0,&P_3\ne0,\\1,&P_3=0.\end{cases}
\end{equation}
\item If $P_3\ne0$, there is a relation
\begin{equation}\label{eq:motiverelation}
 C_2\tau^2\s=C_1\tau\s+C_0\s,
\end{equation}
where
\[
 C_2\in x^2+Lx+L,\quad
 C_1\in Lx^2-P_3y+Lx+L,\quad
 C_0\in xy+Lx^2+Ly+Lx+L.
\]
\end{enumerate}
\end{prop}
\begin{proof}
\emph{(i).} Put $M_3=\A_L\s+\A_L\tau\s+\A_L\tau^2\s$.
If $P_3\ne0$, equation~\eqref{eq:first} expresses $\tau^3\s$ in
$M_3$. Since $\tau(a v)=a^{(1)}\tau v$ and $L$ is perfect, this makes $M_3$ stable under $\tau$.
It contains $\s$ and therefore all $\tau^i\s$. Thus $M_3=M_\phi$.

For the converse, suppose $P_3=0$ and put
\[
 N=\A_L\s+\A_L\tau\s,\qquad T=M_\phi/N,\qquad
 e=\overline{\tau^2\s},\qquad z=\theta^{q^2}-P_2.
\]
Reducing \eqref{eq:first} modulo $N$ gives $(x-z) \overline{\tau^2\s} = (x-z)e=0$.
Lemma~\ref{lem:quotient}, which does not assume $P_3\ne0$, says that
$T = L e + L \overline{\tau^3\s}$.

If $M_3=M_\phi$, then $M_\phi / N = M_3 / N  = \A_L e$, which means $e$ generates $T$ as an $\A_L$-module, so there is a surjection
\[
 \A_L/(x-z)\longrightarrow T,\qquad \overline a\longmapsto ae.
\]
Both sides have dimension two over $L$: on the left, $1$ and $ y$ form an $L$-basis of the quotient.
The surjection must therefore be an isomorphism.

The divisor of local lengths of $\A_L/(x-z)$ is the zero divisor
$(x-z)_0$, including its multiplicities if the fiber is ramified.
It would follow that the divisor $D$ in Lemma~\ref{lem:quotient}
equals $(x-z)_0$. But
$(x-z)=(x-z)_0-2\infty$, so the group sum of $(x-z)_0$ is $\infty$,
whereas the group sum of $D$ is $V\ne\infty$. This contradiction
proves the converse.

The submodule $\A_Le=M_3/N$ is nonzero, because $e$ is a member
of the displayed $L$-basis of $T$, and it is proper by the preceding
argument. Since $\dim_LT=2$, it has dimension one. The exact sequence
\[
 0\longrightarrow M_3/N\longrightarrow T
 \longrightarrow M_\phi/M_3\longrightarrow0
\]
proves \eqref{eq:three-generator-defect}.

\emph{(ii).} When $P_3\ne0$, eliminate $\tau^3\s$ from
\eqref{eq:first} and its Frobenius consequence.
The resulting coefficients are
\eqref{eq:C2formula}--\eqref{eq:C0formula}; in particular they have
the asserted leading terms.
\end{proof}

Throughout this section $P_3\ne0$. Write
$D=\dot V+\ddot V$, allowing $\dot V=\ddot V$.

\begin{lem}\label{lem:motiverelation}
For the function $h$ of Lemma~\ref{lem:quotient}, one has
\begin{equation}\label{eq:C0divC2}
 \frac{C_0}{C_2}=\frac{f_\alpha h}{h^{(1)}},\qquad
 \tau S_0=-f_\alpha S_0,\qquad S_0=h\s\wedge\tau\s.
\end{equation}
Moreover,
\[
 (h)=V+\infty-\dot V-\ddot V,\qquad
 \dot V\boxplus\ddot V=V,\qquad
 \frac1h=f_\alpha+m-\kappa
\]
for a unique $\kappa\in L$.
\end{lem}
\begin{proof}
The divisor assertion follows from Lemma~\ref{lem:quotient}.
It implies $1/h\in\mathcal L(V+\infty)$, whose basis is
$1,f_\alpha$. Since $h$ has sign one, $1/h=f_\alpha+m-\kappa$.

On the one hand, the rank-one motive $V^{-1} \A_L$ gives $\tau S_0=c f_\alpha S_0$ for some $c\in L^\times$.
On the other hand, taking wedges in \eqref{eq:motiverelation} gives $\tau S_0=-h^{(1)}C_0/(hC_2)\,S_0$.
The ratio $C_0/C_2$ has sign one, as do $h,h^{(1)},f_\alpha$.
Thus $c=-1$, which proves \eqref{eq:C0divC2}.
\end{proof}

\begin{remark}
The surjection $\A_L^3\to M_\phi$ defined by
$\s,\tau\s,\tau^2\s$ has a projective kernel of rank one, since
$M_\phi$ is projective of rank two. Equation~\eqref{eq:motiverelation}
gives an explicit element of this relation module. The relation module
need not be free over the Dedekind ring $\A_L$.
\end{remark}

Define the following functions and scalars:
\begin{align}
 \mathfrak L_\kappa&=y-\beta^\vee-\kappa(x-\alpha),&
 h&=\frac{x-\alpha}{\mathfrak L_\kappa},\label{eq:hfa}\\
 C_3&=y+(\kappa+a_1)(x-\alpha)-\beta,&
 \gamma_2&=-\kappa^q-m-a_1,\label{eq:C3formula}\\
 \gamma_1&=(\kappa^q+m+a_1)(m-\kappa)-\theta,&
 \chi_1&=\kappa^2+a_1\kappa-a_2-\alpha,\label{eq:aux1}\\
 \chi_2&=a_4-2\kappa\beta^\vee+\kappa^2\alpha
         -a_1\beta^\vee-a_3\kappa+a_2\alpha+\alpha^2, & S&=\lambda_1^q\lambda_2+\gamma_1^q.
 \label{eq:defchi2}
\end{align}
They make sense for every $\kappa$, whether or not it comes from a module.

\begin{prop}\label{cor:C0C1C2}
The coefficients in \eqref{eq:motiverelation} are
\begin{equation}\label{eq:geometricC}
 \begin{gathered}
 C_2=x^2-\chi_1x+\chi_2=(x-\dot\alpha)(x-\ddot\alpha),\qquad C_3=hC_2,\\
 C_1=\lambda_2C_2+\lambda_1C_3,\qquad
 C_0=(x+\gamma_1)C_3+\gamma_2C_2,\qquad \lambda_1=-P_3.
 \end{gathered}
\end{equation}
Their divisors satisfy
\begin{equation}\label{eq:Cdivisors}
 (C_2)_0=D+D^\vee,\quad
 (C_3)=D^\vee+V-3\infty,\quad
 (C_0)_0=\xi+D^{(1)}+D^\vee.
\end{equation}
All multiplicities are included.
\end{prop}
\begin{proof}
Taking wedges of the three generators $\s, \tau \s$ and $\tau^2 \s$ gives
\begin{equation}\label{eq:wedgeideal}
 \langle1,C_1/C_2,C_0/C_2\rangle_{\A_L}=D^{-1}\A_L.
\end{equation}
Hence $D\le(C_2)_0$. 
Since $\deg(D)=2$, a monic quadratic in $x$ with this property is necessarily $(x-\dot\alpha)(x-\ddot\alpha)$.

For the multiplicity check, two distinct points of $D$ cannot be
inverses, since their sum is $V\ne\infty$, so their $x$-coordinates
are distinct. If $D=2P$, then $2P=V\ne\infty$. Hence $P$ is not
2-torsion and $x-x(P)$ is a uniformizer at $P$. 
Vanishing to order two
requires the square of that factor. 
A simple point which is 2-torsion causes no problem: $x-x(P)$ has order two there. 
This proves the first divisor formula. 

Substitution of the line $\mathfrak L_\kappa = 0 $ into the Weierstrass equation gives $\chi_1,\chi_2$. 
Applying $y^\vee=-y-a_1x-a_3$ to the line
$\mathfrak L_\kappa$ gives precisely $C_3$~\eqref{eq:C3formula}. 
Since $\mathfrak L_\kappa$ passes through $V^\vee$ and $P \in D$, its inverse $C_3$ passes through $V$ and $P^\vee \in D^\vee$. 
Comparing $(C_3)$ and $(hC_2)$ gives $C_3=hC_2$ as displayed, since both functions have sign one.

By \eqref{eq:wedgeideal} and $(C_2)_0=D+D^\vee $, $C_1$ vanishes along $D^\vee$.
Subtract its $x^2$ coefficient times $C_2$. 
The remainder has pole order at most three, vanishes along $D^\vee$, and has $y$ coefficient $-P_3$.
The space $\mathcal L(3\infty-D^\vee)$ has dimension one and is spanned
by $C_3$. 
This proves the formula for $C_1$.


Finally, \eqref{eq:xfalpha} and $1/h=f_\alpha+m-\kappa$, together
with the definitions of $\gamma_1,\gamma_2$, give
\begin{equation}\label{eq:basic-rational}
 \frac{f_\alpha}{h^{(1)}}
 =f_\alpha f_\alpha^{(1)}+(m^q-\kappa^q)f_\alpha
 =x-\theta+\gamma_2f_\alpha
 =x+\gamma_1+\frac{\gamma_2}{h}.
\end{equation}
Multiplying by $h$ and using \eqref{eq:C0divC2} and $C_3=hC_2$ yields
\[
 C_0=(x+\gamma_1)C_3+\gamma_2C_2.
\]
Moreover,
\[
 (C_0)=(C_2)+(h)+(f_\alpha)-(h^{(1)})
 =\xi+D^{(1)}+D^\vee-5\infty,
\]
which proves the stated zero divisor.
\end{proof}

\subsection{The fractional-ideal decomposition}
Set
\begin{equation}\label{eq:DefE1E2}
 E_1=\s,\qquad E_2=h\bigl(\lambda_1\tau\s+(x+\gamma_1)\s\bigr).
\end{equation}
Then \eqref{eq:geometricC} and and $1/h=f_\alpha+m-\kappa$ gives
\begin{equation}\label{eq:E2oftau}
 E_2=(\tau^2-\lambda_2\tau-\gamma_2)\s,\qquad
 f_\alpha E_2=\lambda_1\tau\s+(x+\gamma_1)\s+(\kappa-m)E_2.
\end{equation}
Thus $E_2,f_\alpha E_2\in M_\phi$. 
Conversely, these identities express
$\s,\tau\s,\tau^2\s$ in the lattice below. Since $\lambda_1\ne0$,
$E_1,E_2$ are independent over $\mathbf K_L$. Consequently
\begin{equation}\label{eq:Mphidecomposition}
 M_\phi=\A_LE_1\oplus V^{-1}\A_LE_2.
\end{equation}
The second summand is the invertible fractional ideal $V^{-1}\A_L$.

For use in the converse construction, take independent symbols $\s,t$ over $\mathbf K_L$ and prescribe
\begin{equation}\label{eq:rationaltau}
 \tau\s=t,\qquad
 \tau t=\bigl((x+\gamma_1)h+\gamma_2\bigr)\s
              +(\lambda_2+\lambda_1h)t.
\end{equation}
Use \eqref{eq:DefE1E2} with $t=\tau\s$ to define the lattice $M=\A_LE_1\oplus V^{-1}\A_LE_2$.
Here $M$ is merely an abstract lattice with a prescribed \(\tau\)-action and need not arise from a Drinfeld module.

A calculation gives
\begin{equation}\label{eq:tauE1andE2}
 \tau\binom{E_1}{E_2}
 =\frac1{\lambda_1}
 \begin{pmatrix}-(x+\gamma_1)&1/h\\T_1&T_2\end{pmatrix}
 \binom{E_1}{E_2},
\end{equation}
where, writing $S=\lambda_1^q\lambda_2+\gamma_1^q$ ,
\begin{align}
 T_1&=h^{(1)}\bigl(\lambda_1^{q+1}\gamma_2-(x+S)(x+\gamma_1)\bigr),
 \notag\\
 T_2&=h^{(1)}\bigl(\lambda_1^{q+1}+(x+S)/h\bigr).\label{eq:DefT}
\end{align}

\subsection{Stability, including a repeated divisor}
Put
\begin{equation}\label{eq:uv}
 \uk=\kappa^q-\kappa,\qquad
 \vk=(\beta^\vee)^q-\beta^\vee-\kappa^q\alpha^q+\kappa\alpha.
\end{equation}
Two polynomial identities used below hold without any stability assumption (see Appendix~\ref{app:auxiliaryidentities}):
\begin{equation}\label{eq:unconditional}
 \uk\gamma_2+\chi_1^q+\gamma_1=\alpha,\qquad
 \chi_2^q=\gamma_2\vk-\alpha\gamma_1.
\end{equation}

\begin{lem}\label{lem:stable}
Let $\lambda_1\ne0$. The lattice $M$ with the action
\eqref{eq:rationaltau} is $\tau$-stable if and only if
\begin{align}
 -\chi_1^q&=S+\gamma_1,&
 \chi_2^q&=S\gamma_1-\lambda_1^{q+1}\gamma_2,
 \label{eq:con_A1}\\
 -\uk\chi_1^q&=\lambda_1^{q+1}+\uk S+\vk,&
 \uk\chi_2^q&=S\vk-\lambda_1^{q+1}\alpha.
 \label{eq:con_B1}
\end{align}
Equivalently, it is stable if and only if
\begin{equation}\label{eq:con_C1}
 \lambda_1^{q+1}=\uk\gamma_1-\vk,\qquad
 S = -\gamma_1 - \chi_1^q.
\end{equation}
These statements include $D=2P$.
\end{lem}
\begin{proof}
The identities \eqref{eq:E2oftau} hold formally in the abstract
construction defined by \eqref{eq:rationaltau} and \eqref{eq:DefE1E2}.
They show that $\tau E_1$ and $\tau^2E_1$ belong to $M$, and that
\[
 \tau(f_\alpha E_2)=\lambda_1^q\tau^2E_1
 +(x+\gamma_1^q)\tau E_1+(\kappa^q-m^q)\tau E_2.
\]
Thus stability is equivalent to $T_1\in\A_L$ and
$T_2\in V^{-1}\A_L$.

\emph{Step 1.}
Put $\Gamma_1=-T_1/h^{(1)}$ and
$\Gamma_2=(x-\alpha)T_2/h^{(1)}$.
The first is a monic quadratic in $x$. Since $h^{(1)}$ has pole divisor
$D^{(1)}$, regularity of $T_1$ is equivalent to vanishing of $\Gamma_1$
along that divisor, with multiplicities. The multiplicity argument in
Proposition~\ref{cor:C0C1C2} gives $\Gamma_1=C_2^{(1)}$.
Comparing coefficients proves \eqref{eq:con_A1}.

By definition and \eqref{eq:hfa}, the expression for $\Gamma_2$ is   
\begin{equation}\label{eq:DefGam2}
 \Gamma_2=\lambda_1^{q+1}(x-\alpha)+\mathfrak L_\kappa(x+S).
\end{equation}
If $T_2\in V^{-1}\A_L$, its divisor shows that $\Gamma_2$ vanishes along $D^{(1)}+V^\vee$. 
In fact
\[
 \left(\frac{h^{(1)}}{x-\alpha}\right)
 =V^{(1)}-V-V^\vee-D^{(1)}+3\infty,
\]
and $V^{(1)}$ belongs to neither $D^{(1)}$ nor $V^\vee$.
We explain the vanishing along $D^{(1)}$ by polynomial division.
Here ``vanishes along'' means having order at least the multiplicity of each point in the divisor. 
The line
\[
 \mathfrak L_\kappa^{(1)}
 =y-(\beta^\vee)^q-\kappa^q(x-\alpha^q)
\]
vanishes along $D^{(1)}$. Since
$\mathfrak L_\kappa=\mathfrak L_\kappa^{(1)}+\uk x+\vk$,
equation~\eqref{eq:DefGam2} gives
\begin{align*}
 \Gamma_2&=(x+S)\mathfrak L_\kappa^{(1)}+R(x),\\
 R(x)&=\lambda_1^{q+1}(x-\alpha)+(x+S)(\uk x+\vk).
\end{align*}
Thus $\Gamma_2$ vanishes along $D^{(1)}$ if and only if $R(x)$ does.
Divide this polynomial by the monic quadratic $C_2^{(1)}$:
\[
 R(x)=\uk C_2^{(1)}(x)+r_1x+r_0,
\]
where
\[
 r_1=\lambda_1^{q+1}+\uk S+\vk+\uk\chi_1^q,\qquad
 r_0=S\vk-\lambda_1^{q+1}\alpha-\uk\chi_2^q.
\]
If $D^{(1)}=P_1+P_2$ has distinct points, their $x$-coordinates
are distinct: otherwise they would be inverse points and their
group sum would be $\infty$, rather than $V^{(1)}$.
Consequently $r_1x+r_0$ vanishes at both points only when
$r_1=r_0=0$.

If $D^{(1)}=2P$, then $2P=V^{(1)}\ne\infty$, so $P$ is not
2-torsion. Hence $u=x-x(P)$ is a local parameter at $P$.
The remainder becomes
\[
 r_1x+r_0=(r_1x(P)+r_0)+r_1u.
\]
Vanishing to order at least two means that both the constant and
the $u$ coefficient are zero. Again $r_1=r_0=0$.
This proves the necessity of \eqref{eq:con_B1} in both cases,
without replacing a double zero by a single evaluation.

\emph{Step 2.}
Subtracting $\uk$ times the first equation of \eqref{eq:con_A1} from
the first of \eqref{eq:con_B1} gives \eqref{eq:con_C1}.
Conversely assume \eqref{eq:con_C1}. 
Then $S=\uk\gamma_2-\alpha$ by \eqref{eq:defchi2} and \eqref{eq:unconditional}.
The remaining three equations follow from the explicit identities
\begin{align*}
 S\gamma_1-\lambda_1^{q+1}\gamma_2
  &=(\uk\gamma_2-\alpha)\gamma_1
       -(\uk\gamma_1-\vk)\gamma_2
    =\gamma_2\vk-\alpha\gamma_1=\chi_2^q,\\
 \lambda_1^{q+1}+\uk S+\vk
  &=\uk(\gamma_1+S)=-\uk\chi_1^q,\\
 S\vk-\lambda_1^{q+1}\alpha
  &=(\uk\gamma_2-\alpha)\vk-(\uk\gamma_1-\vk)\alpha
    =\uk\chi_2^q.
\end{align*}

\emph{Step 3.}
Assume \eqref{eq:con_C1}, so that \eqref{eq:con_A1} and \eqref{eq:con_B1} also hold. 
We claim that
\begin{equation}\label{eq:T2regular}
 T_1=-C_3^{(1)},\qquad
 T_2=\frac{(x+S)(x-\alpha^q)+\uk C_3^{(1)}}{x-\alpha}.
\end{equation}

First, the definition of $\Gamma_1$ and
\eqref{eq:con_A1} give
\[
 \Gamma_1=(x+S)(x+\gamma_1)-\lambda_1^{q+1}\gamma_2
 =x^2-\chi_1^qx+\chi_2^q=C_2^{(1)}.
\]
Twisting $C_3=hC_2$ therefore yields
\[
 T_1=-h^{(1)}\Gamma_1=-h^{(1)}C_2^{(1)}=-C_3^{(1)} \in \A_L.
\]

For $T_2$, equations \eqref{eq:con_B1} say precisely that
$r_1=r_0=0$ in the division above. Hence
\[
 \Gamma_2=(x+S)\mathfrak L_\kappa^{(1)}+\uk C_2^{(1)}.
\]
Twisting $h=(x-\alpha)/\mathfrak L_\kappa$ and $C_3=hC_2$ gives
\[
 h^{(1)}\mathfrak L_\kappa^{(1)}=x-\alpha^q,
 \qquad h^{(1)}C_2^{(1)}=C_3^{(1)}.
\]
Substitution into $T_2=h^{(1)}\Gamma_2/(x-\alpha)$ proves the second
formula in \eqref{eq:T2regular}.
At $V^\vee=(\alpha,\beta^\vee)$, the definitions of $m$ and
$\gamma_2$ give
\[
 C_3^{(1)}(V^\vee)
 =\beta^\vee-\beta^q+(\kappa^q+a_1)(\alpha-\alpha^q)
 =(m+\kappa^q+a_1)(\alpha-\alpha^q)
 =-\gamma_2(\alpha-\alpha^q).
\]
Since $\alpha+S=\uk\gamma_2$ from \emph{(ii)}, the numerator of $T_2$ in \eqref{eq:T2regular} vanishes at $V^\vee$:
\[
 \big((x+S)(x-\alpha^q)+\uk C_3^{(1)} \big)(V^\vee) =(\alpha+S-\uk\gamma_2)(\alpha-\alpha^q)=0.
\]
The numerator of $T_2$ is regular at every finite point, while $x-\alpha$ has simple zeros at the two distinct points $V$ and $V^\vee$. 
The zero of the numerator at $V^\vee$ cancels the corresponding zero of the denominator.
Consequently, $T_2$ is regular at every finite point except possibly $V$, where its pole has order at most one. 
This proves $T_2\in V^{-1}\A_L$ and completes the sufficiency argument.

\end{proof}


\subsection{The pole filtration and the converse construction}
We next show that every stable lattice constructed above is freely
generated by $\s$ over $L\{\tau\}$. The proof uses a filtration by
pole order and applies equally to reduced and repeated divisors.

For $n \geqslant 1$, define
\begin{equation}\label{eq:filtration}
 \mathcal F_n=
 \left(\mathcal L\left(D+\left\lfloor\frac{n-1}{2}\right\rfloor\infty\right)\s
 \oplus
 \mathcal L\left(D+\left\lfloor\frac{n-2}{2}\right\rfloor\infty\right)t
 \right)\cap M.
\end{equation}
The change of generators gives the membership conditions
\begin{equation}\label{eq:membership}
 a,b\in D^{-1}\A_L,\qquad
 a-\lambda_1^{-1}(x+\gamma_1)b\in\A_L.
\end{equation}
These conditions hold if and only if $a\s+bt\in M$.

\begin{thm}\label{thm:shtukadata}
Assume \eqref{eq:con_C1} and $\lambda_1\ne0$. Then
\[
 \mathcal F_n=\Span_L\{\s,\tau\s,\ldots,\tau^{n-1}\s\},\qquad
 \dim_L\mathcal F_n=n.
\]
Consequently $M=L\{\tau\}\s$ is free of rank one, and
$\mathcal F_{n+1}=\mathcal F_n+\tau\mathcal F_n$.
The proof holds for both reduced and repeated $D$.
\end{thm}
\begin{proof}
Write $\tau^i\s=u_i\s+v_it$. Since $\ord_\infty h=1$, the
coefficient of $\s$ in \eqref{eq:rationaltau} has order $-1$ and sign
one at infinity, while the coefficient of $t$ has order at least zero.
The recursions
\[
 u_{i+1}=v_i^{(1)}\bigl((x+\gamma_1)h+\gamma_2\bigr),\qquad
 v_{i+1}=u_i^{(1)}+v_i^{(1)}(\lambda_2+\lambda_1h)
\]
therefore give
\[
 \ord_\infty(u_{i+1})=\ord_\infty(v_i)-1,\qquad
 \ord_\infty(v_{i+1})\geq
 \min\{\ord_\infty(u_i),\ord_\infty(v_i)\}.
\]
Starting from $u_1=\infty$, $v_1=0$ and $u_2=-1$, $v_2=0$ , induction gives
\[
 \ord_\infty(u_{2k})=-k,\qquad
 \ord_\infty(v_{2k})\geq-(k-1),
\]
with $u_{2k}$ of sign one, and
\[
 \ord_\infty(v_{2k+1})=-k,\qquad
 \ord_\infty(u_{2k+1})\geq-k,
\]
with $v_{2k+1}$ of sign one. Indeed, in the second recursion, the term
$u_{2k}^{(1)}$ supplies the leading term of $v_{2k+1}$, while
$v_{2k}^{(1)}(\lambda_2+\lambda_1h)$ has strictly higher order and
cannot cancel it.

Thus the leading $\s$-terms at even indices and the leading $t$-terms at odd indices prevent cancellation in any nontrivial linear combination. 
Hence the iterates $\{ \tau^i \s \}$ are $L$-independent. The order bounds and \eqref{eq:membership} show that the first $n$ iterates belong to $\mathcal F_n$.

It remains to compute the dimension $\dim_L\mathcal F_n$ without assuming cyclicity.
Let $W_n$ be the direct sum inside parentheses in \eqref{eq:filtration}.
By Riemann--Roch, $\dim_LW_n=n+2$, also for $n=1$.
Let $\mathcal P_D=H^0(E_L,\mathcal O(D)/\mathcal O)$ be the space of
principal parts at $D$; it has $L$-dimension two. By \eqref{eq:membership},
$\mathcal F_n$ is the kernel of
\[
 W_n\longrightarrow\mathcal P_D,\qquad
 a\s+bt\longmapsto
 [a-\lambda_1^{-1}(x+\gamma_1)b].
\]
Both $h\s$ and $ht$ belong to $W_n$. If $D=P+R$ is reduced, their
images are independent since $x(P)\ne x(R)$ and $h$ has a nonzero
simple principal part at both points.
If $D=2P$, choose a parameter $z$ at $P$ and write
\[
 h=c_{-2}z^{-2}+c_{-1}z^{-1}+\cdots,\qquad
 x+\gamma_1=e_0+e_1z+\cdots.
\]
Here $c_{-2}\ne0$ and $e_1\ne0$: otherwise $x$ would ramify at $P$,
so $P=P^\vee$ and $2P=\infty$, contrary to $2P=V$.
In the principal-part coordinates $(z^{-2},z^{-1})$, the two images
have determinant $-\lambda_1^{-1}c_{-2}^2e_1\ne0$.
Thus the map is surjective in both cases. Its kernel has dimension
$n$, so it equals the span of the first $n$ iterates.
Finally $\bigcup_n\mathcal F_n=M$, which proves cyclicity and the
remaining assertions.
\end{proof}

\begin{prop}\label{prop:converse}
Every pair satisfying \eqref{eq:con_C1} with $\lambda_1\ne0$ produces
a sign-normalized rank-two Drinfeld module of characteristic $\xi$
and determinant point $V$. For this module $P_3\ne0$, and the parameters
recovered from its motive are the input parameters
$\kappa,\lambda_1,\lambda_2$. In particular $P_3=-\lambda_1$.
\end{prop}
\begin{proof}
We verify the construction, its characteristic, and parameter recovery
in turn.

\emph{The lattice and the linearized Frobenius.}
Start with the rational action \eqref{eq:rationaltau} and the lattice
$M=\A_LE_1\oplus V^{-1}\A_LE_2$.
It is projective of rank two because $V^{-1}\A_L$ is an invertible
fractional ideal. Lemma~\ref{lem:stable} proves $\tau M\subseteq M$.
Theorem~\ref{thm:shtukadata} proves that
$\s,\tau\s,\tau^2\s,\ldots$ is an $L$-basis of $M$.
Thus $M=L\{\tau\}\s$ freely, and this conclusion has been obtained
without assuming that a Drinfeld module already exists.

Write $\Phi$ for the matrix in \eqref{eq:tauE1andE2}.
Using \eqref{eq:DefT} and \eqref{eq:basic-rational}, its determinant is
\begin{align*}
 \det\Phi
 &=\lambda_1^{-2}\bigl(-(x+\gamma_1)T_2-T_1/h\bigr)\\
 &=-\lambda_1^{q-1}h^{(1)}
       \bigl(x+\gamma_1+\gamma_2/h\bigr)
 =-\lambda_1^{q-1}f_\alpha.
\end{align*}
In particular it is nonzero, so the linearization
$\tau:M^{(1)}\to M$ is injective.

For the local cokernel calculation, in the generator $E_1\wedge E_2$,
the target determinant lattice is $V^{-1}\A_L$, the source determinant
lattice is $(V^{(1)})^{-1}\A_L$, and the map is multiplication by
$-\lambda_1^{q-1}f_\alpha$.
If $\mathfrak m_\xi$ is the maximal ideal of the affine point $\xi$,
the divisor of $f_\alpha$ gives the equality of fractional ideals
\[
 f_\alpha(V^{(1)})^{-1}\A_L=\mathfrak m_\xi V^{-1}\A_L.
\]
At a finite point $P$, choose bases of the source and target lattices
over the discrete valuation ring $\mathcal O_{E_L,P}$.
The two elementary divisors of the resulting injective matrix have
nonnegative exponents whose sum is zero for $P\ne\xi$ and one for
$P=\xi$. The local cokernel is therefore zero away from $\xi$ and
is the one-dimensional residue field at $\xi$.
Consequently
\[
 M/\tau M\simeq\A_L/\mathfrak m_\xi\simeq L.
\]
In the cyclic $L$-basis above, $\tau M$ is the span of
$\tau\s,\tau^2\s,\ldots$, so the nonzero class of $\s$ is a basis
of this quotient.

\emph{The Drinfeld module.}
For $a\in\A$, define the unique twisted polynomial $\phi_a$ by
$a\s=\phi_a\s$. Uniqueness gives additivity and $\phi_c=c$ for
$c\in\Fq$. Since multiplication by $a\in\A$ commutes with $\tau$
and with $L$, one has
\[
 ab\s=\phi_b\phi_a\s=\phi_a\phi_b\s.
\]
The freeness over $L\{\tau\}$ implies
$\phi_{ab}=\phi_a\phi_b$, so $\phi$ is an $\Fq$-algebra homomorphism.
Reduction modulo $\tau M$ shows that the constant coefficient
of $\phi_a$ is $a(\xi)$.

Let $a\ne0$ have pole order $d$ at infinity.
The definition of $\mathcal F_n$ gives
$a\s\in\mathcal F_{2d+1}\setminus\mathcal F_{2d}$ when $d>0$.
Indeed its $\s$ coefficient has order $-d$, whereas membership in
$\mathcal F_{2d}$ requires order at least $-(d-1)$.
The leading $\s$ coefficient of $\tau^{2d}\s$ has order $-d$
and sign one, and every earlier iterate has its $\s$ coefficient
of strictly smaller pole order. It follows that
\[
 \deg_\tau\phi_a=2d,\qquad
 \text{leading coefficient of }\phi_a=\Sgn(a).
\]
Constants satisfy the same normalization directly.
Thus $\phi$ has rank two and is sign-normalized.
The determinant point is $V$: multiplying its rational generator by $c\in L^\times$ with $c^{q-1}=-\lambda_1^{1-q}$ changes $-\lambda_1^{q-1}f_\alpha$ to $f_\alpha$.


\emph{Recovery of the parameters.}
The identities \eqref{eq:E2oftau} put both $E_2$ and $f_\alpha E_2$
in $\A_L\s+\A_L\tau\s+\A_L\tau^2\s$.
Since $V^{-1}\A_L=\A_L+\A_Lf_\alpha$, these three elements
generate all of $M$. The equivalence in
Proposition~\ref{prop:motive2} therefore gives $P_3\ne0$ for the
constructed module.

Put $t=\tau\s$. The change of rational basis has determinant
$\lambda_1h$, so
\[
 \bigwedge^2M
 =\lambda_1hV^{-1}\A_L(\s\wedge t)
 =D_\kappa^{-1}\A_L(\s\wedge t).
\]
Here the scalar $\lambda_1$ is a unit of $\A_L$.
By Lemma~\ref{lem:quotient}, the quotient by
$\A_L\s+\A_Lt$ therefore recovers exactly $D_\kappa$.
It also recovers $h$, because $(h)=V+\infty-D_\kappa$ and
$\Sgn^L(h)=1$ determine that function uniquely.
The identity $h^{-1}=f_\alpha+m-\kappa$ then recovers the input
$\kappa$.

Proposition~\ref{cor:C0C1C2} recovers the same monic polynomial
$C_2$ and the function $C_3=hC_2$.
In the rational basis $(\s,t)$, the coefficient of $t$ in
$\tau^2\s$ is uniquely determined and, by construction, equals
$\lambda_2+\lambda_1h$. Hence its recovered numerator is
\[
 C_1=C_2(\lambda_2+\lambda_1h)
     =\lambda_2C_2+\lambda_1C_3.
\]
The coefficient of $y$ is $\lambda_1$, whereas the elimination
formula gives $-P_3$; the coefficient of $x^2$ is $\lambda_2$.
This recovers both remaining parameters and proves
$P_3=-\lambda_1$.
\end{proof}

\section{The parameter curve and the locus \texorpdfstring{$P_3=0$}{P3=0}}\label{sec:Moduli}
\subsection{Explicit coefficients on \texorpdfstring{$P_3\ne0$}{P3 nonzero}}
By \eqref{eq:uv}, the polynomial \eqref{eq:pi-expanded} is
\begin{equation}\label{def:piX}
 \pi(\kappa)=\uk\gamma_1-\vk.
\end{equation}
The following three polynomials in $\kappa$ shorten the coefficient
formulas and will also measure their orders at the missing points:
\begin{align}
 \cN(\kappa)&=-\chi_1^q-\gamma_1^q-\gamma_1,\notag\\
 \cH(\kappa)&=\chi_2+\gamma_1(\gamma_1+\chi_1)=C_2(-\gamma_1),\notag\\
 \cJ(\kappa)&=\beta+(\kappa+a_1)(\alpha+\gamma_1)-\eta^q
                +(\kappa^{q^2}+m^q-\kappa^q)(\theta^q+\gamma_1).
 \label{eq:NHJ}
\end{align}
The calligraphic letters $\cN,\cH$ distinguish these polynomials from the submodule $N$ and the pole-order set $H$.

\begin{thm}\label{thm:Drinfeld}
For fixed $V$, the pairs
$\lambda_1^{q+1}=\pi(\kappa)$, $\lambda_1\ne0$, parametrize precisely the
sign-normalized rank-two modules with $P_3\ne0$ and determinant point $V$.
Given such a pair $(\kappa,\lambda_1)$, define the following
polynomials in $\kappa$:
\begin{equation}\label{eq:theoremaux}
\begin{aligned}
 \gamma_2&=-\kappa^q-m-a_1,\\
 \gamma_1&=(\kappa^q+m+a_1)(m-\kappa)-\theta,\\
 \chi_1&=\kappa^2+a_1\kappa-a_2-\alpha,\\
 \chi_2&=a_4-2\kappa\beta^\vee+\kappa^2\alpha
          -a_1\beta^\vee-a_3\kappa+a_2\alpha+\alpha^2.
\end{aligned}
\end{equation}
Here $\beta^\vee=-\beta-a_1\alpha-a_3$, and $m$ is the slope
specified in Section~\ref{sec:Pre}. The polynomials $\cN,\cH,\cJ$
are defined by \eqref{eq:NHJ}. Furthermore, define
\[
 \lambda_2=\frac{\cN(\kappa)}{\lambda_1^q},\qquad
 P_0=m+\kappa^q-\kappa.
\]
The coefficients of the corresponding module are
\begin{equation}\label{eq:allcoefficients}
\begin{aligned}
 g_1&=-\lambda_2\,\frac{\cH(\kappa)}{\pi(\kappa)}
                         +\frac{\cJ(\kappa)}{\lambda_1^q},\\
 g_2&=(P_0^q+\kappa+a_1)^q+\lambda_2^{q+1}
                         +\frac{\cH(\kappa)}{\pi(\kappa)},\\
 g_3&=-\lambda_2^{q^2}-\lambda_2,\\
 h_1&=P_0g_1-\lambda_1(P_0^q+\kappa+a_1)
                         -\lambda_2(\theta^q-\gamma_1-\chi_1),\\
 h_2&=\lambda_1\lambda_2^q-\lambda_2g_1^q+P_0g_2
                         -\gamma_1-\chi_1+\theta^{q^2},\\
 h_3&=-\lambda_1+g_1^{q^2}-\lambda_2g_2^q+P_0g_3,\\
 h_4&=P_0+g_2^{q^2}-\lambda_2g_3^q,\qquad
 h_5=g_3^{q^2}-\lambda_2.
\end{aligned}
\end{equation}
In particular $P_3=-\lambda_1$.
\end{thm}
\begin{proof}
Necessity follows from the motive relation and Lemma~\ref{lem:stable}.
Proposition~\ref{prop:converse} proves existence for every indicated pair.
Comparison of the coefficients in the motive relation gives
\eqref{eq:allcoefficients}; the calculation is recorded in
Appendix~\ref{app:coeff}. Conversely, $\lambda_1=-P_3$ is recovered
from the coefficients, and the divisor $D$ and its line recover
$\kappa$. Thus the parametrization is injective.
\end{proof}

\begin{prop}\label{prop:j}
For fixed $V$, two modules in Theorem~\ref{thm:Drinfeld} are isomorphic
over $L$ if and only if their values of $\kappa$ agree.
\end{prop}
\begin{proof}
We first describe all isomorphisms compatible with the chosen normalization. 
Two modules $\psi$ and $\phi$ are isomorphic precisely when
$\psi_a=\ell\phi_a\ell^{-1}$ for some $\ell\in L^\times$ and all
$a\in\A$. For a coefficient $b_i$ of $\tau^i$, conjugation gives
\[
 \ell(b_i\tau^i)\ell^{-1}=\ell^{1-q^i}b_i\tau^i.
\]
The leading coefficients of $\phi_x$ and $\phi_y$ both remain one
if and only if
\[
 \ell^{q^4-1}=\ell^{q^6-1}=1.
\]
Since $\gcd(q^4-1,q^6-1)=q^2-1$, this is equivalent to
$\ell\in\mathbb F_{q^2}^\times$.
Put $c=\ell^{1-q}$; then $c^{q+1}=1$.
For these scalars, $\ell^{1-q^{2j}}=1$ and $\ell^{1-q^{2j+1}}=c$. 
Thus all even coefficients are unchanged and all odd coefficients are multiplied by $c$.
In particular, the definitions \eqref{eq:P0P3def} and $h_5=g_3^{q^2}-\lambda_2$ give
\[
 P_0(\psi)=P_0(\phi),\qquad
 P_3(\psi)=cP_3(\phi),\qquad
 \lambda_2(\psi)=c\lambda_2(\phi).
\]

To see that $\kappa$ is unchanged, consider the corresponding
isomorphism of motives
\[
 M_\psi\longrightarrow M_\phi,\qquad
 P\s_\psi\longmapsto P\ell\s_\phi.
\]
It is $\A_L$-linear, since $\psi_a\ell=\ell\phi_a$, and it
commutes with $\tau$. It sends $\s_\psi$ to $\ell\s_\phi$ and
$\tau\s_\psi$ to $\ell^q\tau\s_\phi$.
Hence it identifies the submodules generated by these first two
elements. Their quotients have the same local lengths, so the
divisor $D$ is unchanged.
For fixed $V$, this divisor determines the sign-normalized
function $h$ by $(h)=V+\infty-D$.
The equality $h^{-1}=f_\alpha+m-\kappa$ now shows that the
two values of $\kappa$ agree.

Conversely, suppose $\kappa$ is the same for two modules.
Write their parameters as $(\kappa,\lambda_1)$ and
$(\kappa,\lambda'_1)$. They satisfy
$(\lambda'_1/\lambda_1)^{q+1}=1$.
Set $c=\lambda'_1/\lambda_1$.
The homomorphism
\[
 \mathbb F_{q^2}^\times\longrightarrow\mu_{q+1},\qquad
 \ell\longmapsto\ell^{1-q}
\]
has kernel $\Fq^\times$, so its image has order $q+1$ and it is
surjective. Choose $\ell$ mapping to $c$.
The quantities $\gamma_i,\chi_i,P_0,\cN,\cH,\cJ$ depend only on
$\kappa$ and $V$, and therefore are unchanged. Moreover
\[
 \lambda'_2=\frac{\cN}{(c\lambda_1)^q}
           =c^{-q}\lambda_2=c\lambda_2,
\]
because $c^q=c^{-1}$.
Substitution in \eqref{eq:allcoefficients} gives
\[
 (g'_1,g'_2,g'_3)=(cg_1,g_2,cg_3),\qquad
 (h'_1,h'_2,h'_3,h'_4,h'_5)
 =(ch_1,h_2,ch_3,h_4,ch_5).
\]
These are exactly the coefficient transformations under conjugation by $\ell$. 
Since $x,y$ generate $\A$, the two modules are isomorphic.
\end{proof}

\begin{proof}[Proof of Theorem~\ref{thm:main}]
\emph{(i).} The generation criterion and the dimension of its defect follow from Proposition~\ref{prop:motive2}.

\emph{(ii).} Theorem~\ref{thm:Drinfeld} and
Proposition~\ref{prop:j} prove this part.
\end{proof}

\subsection{Roots and ramification}\label{subsec:Ramification}
Here $\kappa$ is an arbitrary scalar, not necessarily a parameter of a Drinfeld module. 
Define $h$ by \eqref{eq:hfa}, or equivalently by $1/h=f_\alpha+m-\kappa$. 
It has divisor $V+\infty-D_\kappa$. 
All the geometric formulas for $C_2,C_3,\gamma_i,\chi_i$ \eqref{eq:hfa}--\eqref{eq:geometricC} remain valid by direct line intersection, since they are defined directly in terms of $\kappa$ and the fixed data.
In particular the identity \eqref{eq:basic-rational} is valid for every $\kappa$.

\begin{lem}\label{lem:lambda1eq0}
Let $k$ be a root of $\pi$, and evaluate the auxiliary quantities at $k$.
Then
\begin{equation}\label{eq:rootfactor}
 C_2^{(1)}(x)=(x+\gamma_1)(x+S),
\end{equation}
where $S := -\chi_1^q-\gamma_1$.
\end{lem}
\begin{proof}
 By \eqref{eq:geometricC}, twisting $C_2$ gives $C_2^{(1)}(x)=x^2-\chi_1^qx+\chi_2^q$.
Since $S=-\chi_1^q-\gamma_1$, we have $-\chi_1^q = \gamma_1+S$. 
The first identity in \eqref{eq:unconditional} then gives $S=\uk\gamma_2-\alpha$.

Since $\pi(k)=0$, it follows from \eqref{def:piX} that $\uk\gamma_1=\vk$ at $k$. Consequently,
\[
 \chi_2^q=\gamma_2\vk-\alpha\gamma_1
         =(\uk\gamma_2-\alpha)\gamma_1=S\gamma_1,
\]
where the first equality is the second identity in \eqref{eq:unconditional}. 
Thus $C_2^{(1)}(x)$ and
$(x+\gamma_1)(x+S)$ are monic quadratics with the same coefficients, which proves \eqref{eq:rootfactor}.
\end{proof}

\begin{prop}\label{prop:nomultiple}
\begin{enumerate}
\item The polynomial $\pi$ has exactly the roots of types~\textup{(I)} and
\textup{(II)} in Theorem~\ref{thm:paramcurve}, and each is simple.
\item At type~\textup{(II)} roots, $\cN(k)=0$; at type~\textup{(I)} roots,  $\cN(k)\ne0$.
\end{enumerate}
\end{prop}
\begin{proof}
\emph{(i).} We first produce roots by intersecting $D_\kappa$ and $D_\kappa^{(1)}$.
If a point $T$ belongs to both, then $(1/h)(T)=(1/h^{(1)})(T)=0$.
It is neither $V$ nor $\infty$, so $f_\alpha$ is regular there.
Equation~\eqref{eq:basic-rational} implies $x(T)=-\gamma_1$.
The two lines also give
\[
 0=(\mathfrak L_\kappa-\mathfrak L_\kappa^{(1)})(T)
   =\uk x(T)+\vk.
\]
Thus $\pi(\kappa)=\uk\gamma_1-\vk=0$.

For each $R\in E(\Fq)\setminus\{\infty\}$, there is a unique $\kappa$ such that $D_\kappa = R+(V\boxminus R)$, since the slope $\kappa$ of the line passing through $V^\vee$ is uniquely determined by $R$.
Its points are finite and distinct and it avoids $V$: any failure would force $V \in  E(\overline{\Fq})$.
It and its Frobenius twist contain $R$, so its slope is a root.
Different $R$ give different divisors, since otherwise both points would be rational and their sum $V$ would be rational.
This produces $\#E(\Fq)-1$ distinct roots.

The isogeny $1+\Fr$ is separable. 
With $a_E=q+1-\#E(\Fq)$, its degree is
\[
 \deg(1+\Fr)=q+1+a_E=2(q+1)-\#E(\Fq).
\]
Hence there are exactly this many solutions of $Q\boxplus Q^{(1)}=V$ with $Q \notin  E(\overline{\Fq})$.
The corresponding divisor $D_\kappa = Q + Q^{(1)}$ and its twist contain $Q^{(1)}$, so that slope is a root. 
If two such divisors coincide, either the $Q$'s agree, or they are interchanged by Frobenius. 
The latter gives $Q^{(2)}=Q$, making $Q \in E(\overline{\Fq})$ and then $V \in E(\overline{\Fq})$ . 
Thus these roots are distinct. 
They are disjoint from the first set, since neither $Q$ nor $Q^{(1)}$ is rational.


We have exhibited
\[
 (\#E(\Fq)-1)+(2(q+1)-\#E(\Fq))=2q+1
\]
distinct roots of a polynomial of degree $2q+1$. 
These are all its roots, and they are simple.

\emph{(ii).} 
At a root of type~\textup{(II)}, the common point in $D_\kappa$ and $D_\kappa^{(1)}$ is $Q^{(1)}$, so the calculation at the beginning of part~\emph{(i)} gives
$x(Q^{(1)})=-\gamma_1$.
The other root of $C_2^{(1)}$ is $x(Q^{(2)})=-\gamma_1^q$.
By \eqref{eq:rootfactor}, $S=\gamma_1^q$. Hence $\cN(k)=S-\gamma_1^q=0$ from \eqref{eq:NHJ}.

At a root of type~\textup{(I)}, the common point is $R$, so $\gamma_1=-x(R)\in\Fq$. 
The other point of $D^{(1)}$ is nonconstant and its $x$-coordinate cannot equal $x(R)$, because then it too would be
constant. 
Consequently, $S\ne\gamma_1^q$ and $\cN(k)\ne0$.
\end{proof}

\begin{cor}\label{cor:genus}
The curve $C_\pi:Y^{q+1}=\pi(X)$ is smooth and geometrically irreducible.
Its smooth projective completion $\bar C_\pi$ has genus $q^2$, one point over
infinity, and total ramification over each finite root and infinity.
\end{cor}
\begin{proof}
Since $\pi$ has only simple roots and $p\nmid q+1$, the curve is smooth.
The valuation of $\pi(X)$ at any root is one, which implies that $Y^{q+1}-\pi(X)$ is Eisenstein over the corresponding discrete valuation ring. 
Hence it is irreducible.
Since $\gcd(q+1,2q+1)=1$, there is one point over infinity and it is totally ramified. 
All ramification is tame. 
Riemann--Hurwitz gives
\[
 2g-2=-2(q+1)+(2q+2)q=2q^2-2.\qedhere
\]
\end{proof}

\begin{proof}[Proof of Theorem~\ref{thm:paramcurve}]
\emph{(i).} Proposition~\ref{prop:nomultiple} proves this part.

\emph{(ii).} Corollary~\ref{cor:genus} proves this part.
\end{proof}

\begin{remark}
At a root of $\pi$, one has $\cN(k)=S-\gamma_1^q$.
Thus the factorization \eqref{eq:rootfactor} has $S=\gamma_1^q$
precisely at the type~(II) roots. On $C_\pi^\circ$ the identity
$\lambda_1^q\lambda_2=\cN(\kappa)$ relates this distinction to the
orders of the coefficient functions. Their extensions are determined
in Subsection~\ref{subsec:extension}.
\end{remark}

\subsection{Complete classification when \texorpdfstring{$P_3=0$}{P3=0}}\label{subsec:P3zero}
When $P_3=0$, equation~\eqref{eq:first} gives a quadratic relation
with coefficients of smaller pole order:
\begin{equation}\label{eq:defB}
 \begin{aligned}
 B_2&=x-\theta^{q^2}+P_2,\\
 B_1&=-P_1-(h_5-g_3^{q^2})(x-\theta^q),\\
 B_0&=y-\eta-P_0(x-\theta).
 \end{aligned}
\end{equation}
Equation~\eqref{eq:first} becomes
\begin{equation}\label{eq:Brelation}
 B_2\tau^2\s=B_1\tau\s+B_0\s.
\end{equation}
Every rank-two motive is generated over $\A_L$ by $\s,\tau\s,\tau^2\s,\tau^3\s$, since the equation for $\phi_x$ reduces $\tau^4\s$, and then all higher iterates.

\begin{prop}\label{prop:P3zeroclass}
Suppose that $\phi$ is a sign-normalized rank-two Drinfeld module in generic characteristic with $P_3=0$ and determinant point $V$.
\begin{enumerate}
\item The coefficient $B_1$ vanishes. And there is a unique point $Q$ with
$Q\boxplus Q^{(1)}=V$ such that
\begin{equation}\label{eq:boundarymotive}
 M_\phi=Q^{-1}\A_L\s\oplus(Q^{(1)})^{-1}\A_L\tau\s,\qquad
 \tau^2\s=\hat f\,\s,
\end{equation}
where
\[
 (\hat f)=\xi+Q^{(2)}-Q-\infty,\qquad \Sgn^L(\hat f)=1.
\]
\item $\phi_x$ and $\phi_y$ belong to $L\{\tau^2\}$.
\end{enumerate}
\end{prop}
\begin{proof}
\emph{(i).} The full lattice and its Frobenius stability
are used to determine $B_1$ and the decomposition simultaneously.

\emph{Step 1.}
Write $t=\tau\s$ and then
\[
 a=B_0/B_2,\qquad b=B_1/B_2,\qquad \tau t=a\s+bt.
\]
The two vectors $\s,t$ are independent over $\mathbf K_L$ by
Lemma~\ref{lem:quotient}.
Taking their wedges with $\tau^2\s$ and using that lemma gives
\[
 b(\s\wedge t)=\s\wedge\tau^2\s
 \in D^{-1}\A_L(\s\wedge t),\qquad
 b\in D^{-1}\A_L=hV^{-1}\A_L.
\]
Write $B_2=x-z$. Since $B_1$ has degree at most one in $x$,
polynomial division gives $b=d+e/(x-z)$ with $d,e\in L$.
The constant $d$ also belongs to $D^{-1}\A_L$, so
$(b-d)/h\in V^{-1}\A_L$.
If $e\ne0$, then
\[
 \ord_\infty((b-d)/h)=\ord_\infty(e/(x-z))-\ord_\infty h
 =2-1=1.
\]
It would be a nonzero element of $\mathcal L(V-\infty)$.
This space is zero: a nonzero element would have divisor
$\infty-V$, making the degree-zero divisor $V-\infty$ principal,
contrary to $V\ne\infty$ on an elliptic curve.
Thus $e=0$ and $b=d\in L$.

\emph{Step 2.}
Let $v_i\in\mathbf K_L^2$ be the coordinate vector of $\tau^i\s$ with respect to the ordered basis $(\s,t)$.
 By \eqref{eq:Brelation},
\[
 v_0=(1,0),\quad v_1=(0,1),\quad
 v_2=(a,d),\quad v_3=(d^qa,a^{(1)}+d^{q+1}).
\]
Their six determinants are
\[
 \begin{array}{lll}
 \det(v_0,v_1)=1,&\det(v_0,v_2)=d,
     &\det(v_0,v_3)=a^{(1)}+d^{q+1},\\
 \det(v_1,v_2)=-a,&\det(v_1,v_3)=-d^qa,
     &\det(v_2,v_3)=aa^{(1)}.
 \end{array}
\]
Consequently their determinant ideal is
\begin{equation}\label{eq:boundarydet}
 D^{-1}\A_L
 =\langle1,a,a^{(1)},aa^{(1)}\rangle_{\A_L}
 =\langle1,a\rangle_{\A_L}\langle1,a^{(1)}\rangle_{\A_L}.
\end{equation}
Let $A$ be the finite pole divisor of $a$.
At each finite point $P$ the valuation of
$\langle1,a\rangle$ is $\min\{0,\ord_Pa\}$, so
$\langle1,a\rangle_{\A_L}=A^{-1}\A_L$.
Comparing valuations in \eqref{eq:boundarydet} yields
$D=A+A^{(1)}$.
Taking degrees gives $2=2\deg A$, hence $A=Q$ for one finite point.
Thus
\[
 D=Q+Q^{(1)},\qquad Q\boxplus Q^{(1)}=V,
\]
with $Q \notin E(\overline{\Fq})$. 
In particular $Q,\;Q^{(1)}$ and $Q^{(2)}$ are distinct.

The same four vectors give the lattice itself, not just its
determinant. Indeed,
\[
 v_2-dv_1=(a,0),\qquad
 v_3-d^q(v_2-dv_1)-d^{q+1}v_1=(0,a^{(1)}).
\]
Conversely each $v_i$ lies in the span of
$(1,0),(0,1),(a,0),(0,a^{(1)})$.
It follows that
\begin{equation}\label{eq:boundarylattice-beforezero}
 M_\phi=Q^{-1}\A_L\s\oplus(Q^{(1)})^{-1}\A_Lt.
\end{equation}
At this stage, $d$ has not yet been shown to be zero.

\emph{Step 3.}
For any $g\in(Q^{(1)})^{-1}\A_L$,
\[
 \tau(gt)=g^{(1)}a\s+d\,g^{(1)}t.
\]
Frobenius is an automorphism of the algebraically closed field $L$,
so the functions $g^{(1)}$ range over
$(Q^{(2)})^{-1}\A_L$.
The direct decomposition \eqref{eq:boundarylattice-beforezero}
and $\tau$-stability therefore require
\[
 d\,(Q^{(2)})^{-1}\A_L\subseteq(Q^{(1)})^{-1}\A_L.
\]
If $d\ne0$, localize at $Q^{(2)}$.
Since $Q^{(2)}\ne Q^{(1)}$, the left side is $u^{-1}\mathcal O_{E_L,Q^{(2)}}$ and the right side is $\mathcal O_{E_L,Q^{(2)}}$, where $u$ is a uniformizer.
However, a function with a simple pole at $Q^{(2)}$ cannot belong to $\mathcal O_{E_L,Q^{(2)}}$.
This contradicts the inclusion. Hence $d=0$, and therefore $B_1=0$.
The obstruction is a finite pole in the $t$ component; it cannot be canceled by a term in the independent $\s$ component.

\emph{Step 4.}
With $S_0=h\s\wedge t$, taking wedges gives
\[
 \tau S_0=-a\,h^{(1)}/h\,S_0.
\]
The determinant condition gives $\tau S_0=c f_\alpha S_0$ for some $c\in L^\times$.
The numerator $B_0$ has leading term $y$ and the denominator $B_2$ has leading term $x$.
Thus $a$ has pole order one and sign one at $\infty$.
Since $h,h^{(1)},f_\alpha$ also have sign one, comparison of signs
gives
\[
 a=f_\alpha h/h^{(1)},\qquad
 (a)=\xi+D^{(1)}-D-\infty
     =\xi+Q^{(2)}-Q-\infty.
\]
This is the required sign-normalized function $\hat f$.
Its unique finite pole determines $Q$ uniquely.

\emph{(ii).}
The equation $\tau^2\s=\hat f\s$ implies
\[
 \tau^{2j}\s=
   \bigl(\hat f\hat f^{(2)}\cdots\hat f^{(2j-2)}\bigr)\s,
 \qquad
 \tau^{2j+1}\s=
   \bigl(\hat f^{(1)}\hat f^{(3)}\cdots
                  \hat f^{(2j-1)}\bigr)t,
\]
with $\tau^0 \s = \s$ and $\tau^1 \s = t$.
Multiplication of $\s$ by $x$ or $y$ belongs to $\mathbf K_L\s$.
Its expansion in the cyclic $L$-basis $\{ \tau^i \s \}$ has no odd terms: otherwise their sum would give an $\mathbf{K}_L$-linear dependence for $\s$ and $t$. 
Hence $\phi_x,\phi_y\in L\{\tau^2\}$, as asserted.
\end{proof}

\begin{prop}\label{prop:exampleP3}
By Proposition~\ref{prop:nomultiple}, the type~\textup{(II)} roots of $\pi$ arise precisely from the solutions of $Q\boxplus Q^{(1)}=V$.
\begin{enumerate}
\item Every solution of $Q\boxplus Q^{(1)}=V$ gives a module
as in \eqref{eq:boundarymotive}. Write $Q=(x_Q,y_Q)$ and let
$\hat m$ be the slope through $Q^\vee,Q^{(2)},\xi$. Then
\begin{equation}\label{eq:hatf}
 \hat f=\frac{y-\eta-\hat m(x-\theta)}{x-x_Q}
\end{equation}
and
\begin{equation}\label{eq:evencoeff}
 \begin{gathered}
 \phi_x=\theta+g_2\tau^2+\tau^4,\qquad
 \phi_y=\eta+h_2\tau^2+h_4\tau^4+\tau^6,\\
 g_2=\hat m+\hat m^{q^2}+a_1,\qquad
 h_2=\frac{g_2(\eta^{q^2}-\eta)}{\theta^{q^2}-\theta},\\
 h_4=\frac{\eta^{q^4}-\eta+g_2h_2^{q^2}-h_2g_2^{q^2}}
                  {\theta^{q^4}-\theta}.
 \end{gathered}
\end{equation}
\item There are exactly $2(q+1)-\#E(\Fq)$ such modules, and they are
pairwise nonisomorphic.
\end{enumerate}
\end{prop}
\begin{proof}
\rev{\emph{(i).}} Applying $1-\Fr$ gives $Q\boxminus Q^{(2)}=\xi$, so the indicated
line and function have the stated divisor. Define
$\tau\s=t$, $\tau t=\hat f\s$ on the lattice in
\eqref{eq:boundarymotive}. The first summand maps to the second.
The second maps into the first, since
\[
 (\hat f)-Q^{(2)}=\xi-\infty - Q
\]
on the projective curve.
Consequently, $\tau M^{(1)}\subseteq M$, and the quotient $M/\tau M^{(1)}$ has length one at $\xi$ and vanishes at every other finite point.
The products $1,\hat f,\hat f\hat f^{(2)},\ldots$ form the successive Riemann--Roch bases of $Q^{-1}\A_L$, and their twists do the same for  $(Q^{(1)})^{-1}\A_L$. 
Thus this lattice is freely cyclic over $L\{\tau\}$ and defines a rank-two module.
Applying \eqref{eq:rankonecoeff} with $q$ replaced by $q^2$ gives \eqref{eq:evencoeff}.
Its determinant point is $Q\boxplus Q^{(1)}=V$.

\rev{\emph{(ii).}} 
The count was proved in Proposition~\ref{prop:nomultiple}.
Any isomorphism between sign-normalized modules is conjugation by a sign-preserving scalar $\ell\in\mathbb F_{q^2}^{\times}$. 
Such conjugation fixes every even coefficient, so distinct modules in this family cannot be isomorphic.
\end{proof}

\begin{cor}\label{prop:P3lambda1}\label{cor:P3zeromatrix}
For $P_3=0$, the Frobenius matrix in the basis $(\s,\tau\s)$ is
\[
 \begin{pmatrix}0&1\\\hat f&0\end{pmatrix}.
\]
Moreover $M_\phi$ is not generated over $\A_L$ by $\s,\tau\s,\tau^2\s$, whereas these three elements do generate it when $P_3\ne0$. 
In the first case, the quotient by their span has $L$-dimension one, in agreement with
\eqref{eq:three-generator-defect}.
\end{cor}
\begin{proof}
For $P_3=0$, their span $\A_L\s+\A_L\tau\s+\A_L\tau^2\s$ is $Q^{-1}\A_L\s\oplus\A_L\tau\s$, a proper submodule of
\eqref{eq:boundarymotive}.
And their quotient $(Q^{(1)})^{-1}\A_L/\A_L$ has length one.
The other implication follows from Proposition~\ref{prop:motive2}.
\end{proof}

\subsection{Extension of the coefficient functions}\label{subsec:extension}
Here, saying that a coefficient function extends at a point $z \in \bar C_\pi$ means that it belongs to the local ring $\mathcal O_{\bar C_\pi,z}$.
At a finite point, this agrees with regularity on $C_\pi$.

\begin{prop}\label{prop:nolimit}
\begin{enumerate}
\item Under scalar conjugation by $\ell\in\mathbb F_{q^2}^{\times}$,
\[
 P_0\longmapsto P_0,\qquad P_3\longmapsto\ell^{1-q}P_3.
\]
Thus $P_3=0$ is an intrinsic condition on the isomorphism class.
\item The coefficient functions \eqref{eq:allcoefficients} extend at every type~\textup{(II)} point $(k,0)$ of $C_\pi$.
Their values there are the module attached to the corresponding $Q$ in Proposition~\ref{prop:exampleP3}.
\item At a type~\textup{(I)} point, $g_3$ has pole order $q^3$.
At the unique point $P_\infty$ over $X=\infty$, it has pole order $q^5$.
\end{enumerate}
\end{prop}
\begin{proof}
Let $z=(k,0)$ be a finite point with $\pi(k)=0$.
Since $\pi$ is squarefree, $\lambda_1$ is a uniformizer at $z$ and
\begin{equation}\label{eq:localorders}
 \ord_z(\lambda_1)=1,\qquad \ord_z(\kappa-k)=q+1.
\end{equation}

\emph{(i).} The scalar assertion follows directly from \eqref{eq:P0P3def}:
even coefficients are fixed and odd ones have weight $1-q$.

\emph{(ii).} 
Suppose $z$ has type~\textup{(II)}. We know $D_k=Q+Q^{(1)}$, $\cN(k)=0$ and $x(Q^{(1)})=-\gamma_1$ by Proposition~\ref{prop:nomultiple}. 
Hence $\cH(k)=C_2(-\gamma_1)=0$.
We next verify $\cJ(k)=0$, since this is the other possible pole in
$g_1$. At this slope,
\[
 h^{-1}\text{ has zeros }Q+Q^{(1)},\qquad
 f_\alpha h/h^{(1)}=\hat f
\]
by the divisor calculation in \eqref{eq:boundarymotive}.
Since $C_2=(x-x_Q)(x-x(Q^{(1)}))$, identity
\eqref{eq:basic-rational} gives
\[
 \hat f=\frac{\gamma_2 (x-x_Q) + C_3}{x-x_Q}.
\]
The numerator of the right side is the line in \eqref{eq:hatf}.
Its slope is
$-(\gamma_2+\kappa+a_1)=m+\kappa^q-\kappa=P_0$.
Therefore $\hat m=P_0$ at $k$.
Furthermore the equation of $\mathfrak L_k$ gives
\[
 y_{Q^{(1)}}^\vee=\beta+(k+a_1)(\alpha+\gamma_1).
\]
Consequently, $\cJ(k)$ is the value at $(Q^{(1)})^\vee$ of the twisted line $y-\eta^q-P_0^q(x-\theta^q)$.
That line passes through $(Q^{(1)})^\vee$, so $\cJ(k)=0$.

The three functions $\cN,\cH,\cJ$ are polynomials in $\kappa$.
By $\lambda_2=\cN(\kappa)/\lambda_1^q$ in Theorem~\ref{thm:Drinfeld} and \eqref{eq:localorders}, their vanishing at $k$ imply
\[
 \ord_z\lambda_2 \geqslant 1,\qquad
 \ord_z (\cH/\pi) \geqslant 0,\qquad
 \ord_z(\cJ/\lambda_1^q) \geqslant 1.
\]
It follows from \eqref{eq:allcoefficients} that every coefficient is
regular and that all odd coefficients vanish at $z$.

The polynomial identities expressing commutativity and the equation
of $E$ hold on the dense open curve, hence also after this regular
extension. The leading coefficients remain one, so the specialization
is a rank-two Drinfeld module. To identify it, use
\[
 P_2=\lambda_1\lambda_2^q-\gamma_1-\chi_1+\theta^{q^2}.
\]
At $z$ its relation \eqref{eq:Brelation} has
$B_2=x-(\gamma_1+\chi_1)=x-x_Q$, $B_1=0$, and
$B_0=y-\eta-P_0(x-\theta)$. Thus $B_0/B_2=\hat f$, and the
specialization is exactly \eqref{eq:evencoeff}.

\emph{(iii).}
If $z$ has type~\textup{(I)}, Proposition~\ref{prop:nomultiple} gives $\cN(k)\ne0$.
Thus $\ord_z\lambda_2=-q$.
In $g_3=-\lambda_2^{q^2}-\lambda_2$ the two terms have different orders, so $\ord_z g_3=-q^3$.

It remains to examine $P_\infty$. 
The projection has degree $q+1$ and is totally ramified there by Corollary~\ref{cor:genus}.
Consequently $\ord_{P_\infty}(\kappa)=-(q+1)$.
Since $\deg\pi=2q+1$, the equation $\lambda_1^{q+1}=\pi(\kappa)$ gives
\[
 (q+1)\ord_{P_\infty}(\lambda_1)
 =-(2q+1)(q+1),\qquad
 \ord_{P_\infty}(\lambda_1)=-(2q+1).
\]
By \eqref{eq:theoremaux}, the polynomial $\gamma_1$ has leading term $-\kappa^{q+1}$, whereas $\deg\chi_1=2$. 
Hence
\[
 \cN(\kappa)=-\chi_1^q-\gamma_1^q-\gamma_1
          =\kappa^{q^2+q}+\text{terms of lower degree}.
\]
Indeed $q^2+q>2q$ and $q^2+q>q+1$ for $q \geqslant2$, so its leading term cannot cancel. We obtain
\begin{align*}
 \ord_{P_\infty}(\lambda_2)
 &=\ord_{P_\infty}(\cN)-q\ord_{P_\infty}(\lambda_1)\\
 &=-(q^2+q)(q+1)+q(2q+1)=-q^3.
\end{align*}
The two terms in $g_3=-\lambda_2^{q^2}-\lambda_2$ then have orders $-q^5$ and $-q^3$, respectively. They cannot cancel, so
\begin{equation}\label{eq:g3-infinity-order}
 \ord_{P_\infty}(g_3)=-q^5.
\end{equation}
Thus the coefficient functions do not all extend regularly at
$P_\infty$.
\end{proof}

\begin{remark}\label{rm:basischange}
Although the coefficient functions  \eqref{eq:allcoefficients} extend regularly at a type~(II) point, $(E_1,E_2)$ need not extend to a basis here. 
Indeed, by \eqref{eq:DefE1E2}, the transformation from $(\s,\tau\s)$ to $(E_1,E_2)$ has determinant $h\lambda_1$, which vanishes at $\lambda_1=0$. 
Thus the boundary decomposition \eqref{eq:boundarymotive} is not obtained by simply specializing this basis change.
\end{remark}

The preceding results construct Drinfeld modules from roots of $\pi$.
We now return to the modules in $\mathcal X_V$ themselves and describe their automorphism groups and isomorphism classes.
\begin{cor}\label{cor:scalarclasses}
\begin{enumerate}
\item For every $\phi\in\mathcal X_V$,
\[
 \operatorname{Aut}_L(\phi)=
 \begin{cases}
  \Fq^\times,&P_3\ne0,\\
  \mathbb F_{q^2}^{\times},&P_3=0.
 \end{cases}
\]
The second case is equivalent to the condition that scalar
multiplication by $\mathbb F_{q^2}$ commutes with every $\phi_a$.
\item The slope $\kappa$ induces a bijection
\[
 \mathcal X_V/{\simeq}\ \longrightarrow\
 L\setminus\{k:\ k\text{ is a root of type~\textup{(I)}}\}.
\]
\end{enumerate}
\end{cor}
\begin{proof}
\emph{(i).} An invertible element of $L\{\tau\}$ is a nonzero scalar.
Preservation of the leading coefficients of $\phi_x,\phi_y$
restricts it to $\mathbb F_{q^2}^{\times}$, as in Proposition~\ref{prop:j}. 

If $P_3\ne0$, its transformation law in Proposition~\ref{prop:nolimit} forces $\ell^{1-q}=1$, hence $\ell\in\Fq^\times$. 
If $P_3=0$, all odd coefficients vanish by Proposition~\ref{prop:P3zeroclass}.
Hence every element of $\mathbb F_{q^2}$ commutes with $\phi(\A)$.
Conversely, commutation with a scalar in $\mathbb F_{q^2}\setminus\Fq$ forces every odd coefficient to vanish, and then the definition gives $P_3=0$.

\emph{(ii).} Proposition~\ref{prop:j} treats the nonzero fibers of $\lambda_1^{q+1}=\pi(\kappa)$.
Each type~\textup{(II)} slope determines a unique $Q$ by Proposition~\ref{prop:nomultiple} and distinct such $Q$ give nonisomorphic modules by Proposition~\ref{prop:exampleP3}.
By Proposition~\ref{prop:nolimit}, the condition $P_3=0$ is invariant under isomorphism, so the two parts of the classification remain disjoint.
\end{proof}

\begin{remark}\label{rem:rankoneextension}
When $P_3=0$, the module $\phi$ extends by scalars to a rank-one Drinfeld module for $\mathbb F_{q^2}\otimes_{\Fq}\A$, with Frobenius $\tau^2$. 
Indeed, these scalars commute with $\phi(\A)$, and the degrees of $\phi_x,\phi_y$ in $\tau^2$ are two and three.
Their number is
\[
2(q+1)-\#E(\Fq)=q+1+a_E=L_E(-1),
\]
where $L_E(T) = 1-a_ET+qT^2$ is the $L$-polynomial of the genus-one
function field $\KK$. This agrees with the classical count of elliptic
points of $\operatorname{GL}_2(A)$ in
\cite[Corollary~2.6]{MasonSchweizer15}.
\end{remark}

\begin{cor}\label{cor:boundaryclassification}
The largest open subset of $\bar C_\pi$ on which all coefficient
functions in \eqref{eq:allcoefficients} are regular is
\[
 \bar C_\pi\setminus
 \bigl(\{P_\infty\}\cup
       \{(k,0):k\text{ is a type~\textup{(I)} root}\}\bigr).
\]
This open subset contains every type~\textup{(II)} point, and its complement has $\#E(\Fq)$ points.
\end{cor}
\begin{proof}
Theorem~\ref{thm:Drinfeld} and Proposition~\ref{prop:nolimit} prove the regularity at this open set and shows that $g_3$ a pole at every type~\textup{(I)} point and at $P_\infty$.
Hence the displayed open set is the maximal one. 
By Proposition~\ref{prop:nomultiple}, the number of excluded points is $1+(\#E(\Fq)-1)=\#E(\Fq)$.
\end{proof}

The complement of $C_\pi^\circ$ in $\bar C_\pi$ consists of $P_\infty$, the type~\textup{(I)} points, and the type~\textup{(II)} points. 
We call the type~\textup{(I)} points the \emph{finite cusps} and $P_\infty$ the \emph{infinite cusp}; these are precisely the points where the coefficient family fails to be regular. 
The type~\textup{(II)} points are called \emph{elliptic points}. 
The coefficient family extends at these points, and the corresponding modules have extra automorphisms, accounting for the elliptic-point counted in Remark~\ref{rem:rankoneextension}.

\begin{proof}[Proof of Theorem~\ref{thm:boundary}]
\emph{(i).} Propositions~\ref{prop:P3zeroclass},
\ref{prop:exampleP3}, and \ref{prop:nolimit} prove the boundary classification, including completeness of the $P_3=0$ classification.
Corollary~\ref{cor:scalarclasses} proves the assertions about commuting scalars and scalar automorphisms.

\emph{(ii).} Proposition~\ref{prop:nolimit} gives the two pole orders, and Corollary~\ref{cor:boundaryclassification} identifies the maximal common regularity locus and counts its complement.
\end{proof}


\section{Examples}\label{sec:examples}
\subsection{An explicit construction by composition}
\begin{example}
Let $q=2$ and $E:y^2+y=x^3$. Put
$\theta'=\theta^{1/2}$, $\eta'=\eta^{1/2}$, and $m'=m^{1/2}$.
The rank-one module at the characteristic point
$\xi^{(-1)}=(\theta',\eta')$, associated with $V^{(-1)}$, is
\[
 \psi_x=\theta'+U\tau+\tau^2,\qquad
 \psi_y=\eta'+W\tau+Z\tau^2+\tau^3,
\]
where \eqref{eq:rankonecoeff} gives
\[
 U=m'+m,\qquad
 W=\frac{U(\eta-\eta')}{\theta-\theta'},\qquad
 Z=\frac{\eta^2-\eta'+UW^2-WU^2}{\theta^2-\theta'}.
\]
Define $\phi_a=\psi_{a^2}=\psi_a\circ\psi_a$ for $a\in\A$.
This is an $\mathbb F_2$-algebra homomorphism, since $a\mapsto a^2$
is one. Its constant term is $a(\xi)$ and its $\tau$-degree is
$2\deg a$, so it is a sign-normalized rank-two module.
Direct multiplication in $L\{\tau\}$ gives all eight coefficients:
\begin{align*}
 g_1&=U(\theta'+\theta),&
 g_2&=\theta'+\theta^2+U^3,&
 g_3&=U+U^4,\\
 h_1&=W(\eta'+\eta),&
 h_2&=Z(\eta'+\eta^2)+W^3,&
 h_5&=Z+Z^8,\\
 h_3&=\eta'+\eta^4+WZ^2+ZW^4,&
 h_4&=W+W^8+Z^5.&&
\end{align*}
These factored expressions avoid expanding high powers in an algebraic
presentation of $\alpha$ over $\KK$. The construction also verifies
the operator relations: $\phi_x$ and $\phi_y$ commute, and
$\phi_y^2+\phi_y=\phi_x^3$, because $\phi$ is a homomorphism.

\end{example}
\subsection{A repeated divisor in characteristic two}
\begin{example}
Let $q=2$ and $E:y^2+y=x^3$. Take $\kappa=\sqrt{\alpha}$.
The formulas for $C_2$ reduce to
\[
 \chi_1=0,\qquad \chi_2=\sqrt{\alpha},\qquad
 C_2=(x-\alpha^{1/4})^2.
\]
Thus $D_\kappa=2P$ with $x(P)=\alpha^{1/4}$ and $2P=V$.
The point $P$ is nonconstant. Both types of root divisors in
Proposition~\ref{prop:nomultiple} are reduced, so $\pi(\kappa)\ne0$.
Choosing any $\lambda_1$ with $\lambda_1^3=\pi(\kappa)$ therefore
gives a module in Theorem~\ref{thm:Drinfeld}.
This is an example where the repeated-point arguments in
Lemma~\ref{lem:stable} and Theorem~\ref{thm:shtukadata} are required.

\end{example}
\subsection{The two types of finite missing points}
\begin{example}
For $q=2$, formula \eqref{eq:pi-expanded} becomes
\[
 \pi(X)=X^5+(m+1)X^4+a_1X^3
       +(\alpha+a_2+m+a_1)X^2+(\alpha^2+a_2)X
       +\beta^\vee+(\beta^\vee)^2.
\]
All these parameter curves have genus four.
The following table illustrates the two types of roots for five
elliptic curves over $\mathbb F_2$. The rational point counts follow
by substituting $x,y\in\mathbb F_2$ and adding infinity.
Write $E_1,\ldots,E_5$ for the curves in the order of the table.
\begin{center}
\begin{tabular}{lccc}
\toprule
$E$ & $\#E(\mathbb F_2)$ & Type (I) & Type (II)\\
\midrule
$y^2+xy=x^3+1$       &4&3&2\\
$y^2+xy=x^3+x^2+1$   &2&1&4\\
$y^2+y=x^3+x+1$      &1&0&5\\
$y^2+y=x^3$          &3&2&3\\
$y^2+y=x^3+x$        &5&4&1\\
\bottomrule
\end{tabular}
\end{center}
\FloatBarrier
For reference, their specialized parameter polynomials are
\begin{align*}
 \pi_{E_1}(X)&=X^5+(m+1)X^4+X^3+(\alpha+m+1)X^2
                    +\alpha^2X+\beta^\vee+(\beta^\vee)^2,\\
 \pi_{E_2}(X)&=X^5+(m+1)X^4+X^3+(\alpha+m)X^2
                    +(\alpha^2+1)X+\beta^\vee+(\beta^\vee)^2,\\
 \pi_{E_3}(X)&=X^5+(m+1)X^4+(\alpha+m)X^2+\alpha^2X+\alpha^3+\alpha+1,\\
 \pi_{E_4}(X)&=X^5+(m+1)X^4+(\alpha+m)X^2+\alpha^2X+\alpha^3,\\
 \pi_{E_5}(X)&=X^5+(m+1)X^4+(\alpha+m)X^2+\alpha^2X+\alpha^3+\alpha.
\end{align*}
Here $V=(\alpha,\beta)$ and $m$ are computed on the corresponding
curve. These formulas follow by substitution in \eqref{eq:pi-expanded};
in the last three rows $\beta^\vee=\beta+1$ and the Weierstrass
equation simplifies the constant term.

For $y^2+y=x^3$, there are five simple roots of $\pi$.
Three give regular extensions and the three modules with $P_3=0$.
The other two are indexed by $(0,0),(0,1)$ and have poles in $g_3$.
For $y^2+y=x^3+x+1$, all five finite points with $Y=0$ are regular
extension points; there are no type~(I) points.

The curve $y^2+y=x^3+1$ is isomorphic to $y^2+y=x^3$ under the
substitution $x'=x+1$, $y'=y+x$.

\end{example}
\subsection{An example over \texorpdfstring{$\mathbb F_3$}{F3}}
\begin{example}
The smooth curve
\[
 y^2+xy+y=x^3+x^2+x+1
\]
has exactly two $\mathbb F_3$-points, namely $\infty$ and $(2,0)$.
Here $\deg\pi=7$. The type~(I) root is the slope of the line through
$V^\vee$ and $(2,0)$,
\[
 k=-\frac{\beta^\vee}{2-\alpha}.
\]
At this point $g_3$ has pole order $27$. The other six roots correspond
to the six points $Q$ satisfying $Q\boxplus Q^{(1)}=V$; all give
regular extensions with only even Frobenius powers.
The smooth projective curve $Y^4=\pi(X)$ has genus nine.
\end{example}

\appendix
\section{Notation}\label{subsec:notation}
The following tables collect the notation used in the construction.
The coordinates $x,y$ are functions on $E_L$, whereas
$\theta,\eta,\alpha,\beta$ are elements of the coefficient field $L$.
In particular Frobenius twisting fixes $x,y$ and raises elements of
$L$ to their $q$-th powers.

\begingroup
\small
\renewcommand{\arraystretch}{1.17}
\begin{longtable}{@{}>{\raggedright\arraybackslash}p{0.27\textwidth}>{\raggedright\arraybackslash}p{\dimexpr0.73\textwidth-2\tabcolsep\relax}@{}}
\toprule
Symbol & Meaning \\
\midrule
\endfirsthead
\toprule
Symbol & Meaning (continued) \\
\midrule
\endhead
\bottomrule
\endfoot
$E,\ \infty$ & The smooth Weierstrass curve over $\Fq$ and its identity point. Pole orders of $x,y$ at this point are $2,3$. \\
$\A,\ \KK,\ L$ & The affine coordinate ring $\A=\Fq[x,y]$ with the Weierstrass relation; the second copy $\KK=\Fq(\theta,\eta)$ of its function field; an algebraically closed field $L\supset\KK$. \\
$E_L,\ \A_L,\ \mathbf K_L$ & The base change $E\times_{\Spec\Fq}\Spec L$, its affine coordinate ring $\A_L=L\otimes_{\Fq}\A$, and $\mathbf K_L=\operatorname{Frac}(\A_L)$. \\
$\xi=(\theta,\eta)$ & The characteristic point. The homomorphism $\A\to L$ is evaluation at $\xi$ and is injective. \\
$P^\vee,\ P\boxplus R,\ P\boxminus R$ & The inverse, sum, and difference in the group $E(L)$; $P\boxminus R=P\boxplus R^\vee$. For $P=(a,b)$, $P^\vee=(a,-b-a_1a-a_3)$. \\
$D=\sum_P n_PP,\ D^\vee$ & A divisor and its image $\sum_P n_PP^\vee$ under inversion. The expression $P+R$ is a formal divisor, while $P\boxplus R$ is a point. Repetitions in a divisor record multiplicities. \\
$\Fr,\ (-)^{(i)}$ & $\Fr(P)=P^{(1)}$ is the $q$-power Frobenius on points. On a function, $(-)^{(i)}$ twists its coefficients by $c\mapsto c^{q^i}$ and fixes $x,y$. Twisting a divisor or ideal means applying the same operation to its points or functions. Negative twists use inverse Frobenius on the perfect field $L$. \\
$\Supp D,\ \deg D$ & The points with nonzero multiplicity and the degree of $D$. Since $L$ is algebraically closed, a finite effective divisor has degree $\sum_P n_P$. \\
$(g),\ \ord_P(g)$ & The divisor of a nonzero rational function and its order at $P$. Positive order denotes a zero and negative order a pole. The ambient curve is $E_L$ or $\bar C_\pi$, as specified. \\
$\mathcal L(D)$ & The Riemann--Roch space $\{g\in\mathbf K_L:(g)+D \geqslant 0\}\cup\{0\}$. \\
$D^{-1}\A_L$ & The fractional ideal $\bigcup_{n \geqslant 0}\mathcal L(D+n\infty)$. For effective finite $D$, its functions may have finite poles only at $\Supp D$, with orders bounded by the multiplicities of $D$. The notation is not an inverse point in $E(L)$. \\
$\mathcal O_{E_L,P},\ \mathfrak m_P$ & The local ring of $E_L$ at $P$ and, for finite $P$, its maximal ideal in $\A_L$. \\
$\Sgn^L,\ \Sgn$ & The leading-coefficient sign on $\mathbf K_L^\times$ and its restriction to $\A\setminus\{0\}$. The normalization is $\Sgn^L(x)=\Sgn^L(y)=1$. \\
$V=(\alpha,\beta),\ \beta^\vee$ & The fixed determinant point, satisfying $V\boxminus V^{(1)}=\xi$, and the second coordinate $\beta^\vee=-\beta-a_1\alpha-a_3$ of $V^\vee$. \\
$m,\ \mathcal L_m,\  f_\alpha $ & The slope of the line through $V^\vee,V^{(1)}$, $(\mathcal L_m) =V^\vee+V^{(1)}+\xi-3\infty$ and the sign-normalized shtuka function in \eqref{eq:deffa}. Its divisor is $V^{(1)}+\xi-V-\infty$. \\
$L\{\tau\},\ M_\phi,\ \s$ & The twisted polynomial ring with $\tau c=c^q\tau$, the motive $M_\phi=L\{\tau\}\s$, and its cyclic generator. The vector $\s$ corresponds to the constant polynomial $1$. Multiplication by $a\in\A$ is $a(P\s)=P\phi_a\s$. \\
$N,\ D,\ h$ & In Lemma~\ref{lem:quotient}, $N=\A_L\s+\A_L\tau\s$, $D$ records the local lengths of $M_\phi/N$, and $(h)=V+\infty-D$ with $\Sgn^L(h)=1$. Thus $\deg D=2$ and the group sum of $D$ is $V$. \\
$S_0$ & The rational determinant generator $h(\s\wedge\tau\s)$. Its lattice is $\bigwedge^2M_\phi=V^{-1}\A_LS_0$, with Frobenius coefficient $-f_\alpha$ in this normalization. \\
$H,\ \mathcal P_D,\ \mathcal F_n$ & The pole-order set $H=\{0,2,3,\ldots\}$, the principal-part space $\mathcal P_D=H^0(E_L,\mathcal O(D)/\mathcal O)$, and the filtration in \eqref{eq:filtration}. The space $\mathcal P_D$ has dimension $2$ for the divisor $D$ above. \\
\end{longtable}
\endgroup

The parameters and coefficient functions have the following meanings.
We use $\cN(\kappa)$ and $\cH(\kappa)$ for polynomials, while
$N$ and $H$ retain their meanings as a submodule and the pole-order set.
\begingroup
\small
\renewcommand{\arraystretch}{1.17}
\begin{longtable}{@{}>{\raggedright\arraybackslash}p{0.27\textwidth}>{\raggedright\arraybackslash}p{\dimexpr0.73\textwidth-2\tabcolsep\relax}@{}}
\toprule
Symbol & Meaning \\
\midrule
\endfirsthead
\toprule
Symbol & Meaning (continued) \\
\midrule
\endhead
\bottomrule
\endfoot
$g_i,\ h_i,\ P_i$ & The Frobenius coefficients of $\phi_x,\phi_y$ in \eqref{eq:phix2}--\eqref{eq:phiy2}, and the scalars $P_0,P_1,P_2,P_3$ defined in \eqref{eq:P0P3def} and \eqref{eq:P12}. \\
$\kappa,\ D_\kappa,\ \mathfrak L_\kappa$  & A line slope, its residual divisor, and the line function $\mathfrak L_\kappa=y-\beta^\vee-\kappa(x-\alpha)$. Their relation is $(\mathfrak L_\kappa)=V^\vee+D_\kappa-3\infty$. For a module with $P_3\ne0$, $D=D_\kappa$ and $h^{-1}=f_\alpha+m-\kappa$. \\
$\dot V,\ \ddot V$ & The two points of $D_\kappa=\dot V+\ddot V$, allowing equality. Their first coordinates are $\dot\alpha,\ddot\alpha$. \\
$C_0,C_1,C_2,C_3$ & The coefficients of $C_2\tau^2\s=C_1\tau\s+C_0\s$ and the function $C_3=hC_2$, given in \eqref{eq:geometricC}. In particular $C_2=x^2-\chi_1x+\chi_2$. \\
$\lambda_1,\ \lambda_2$ & The scalars in $C_1=\lambda_2C_2+\lambda_1C_3$. On the open family, $\lambda_1=-P_3\ne0$ and $\lambda_2=\cN(\kappa)/\lambda_1^q$. \\
$\gamma_1,\gamma_2,\chi_1,\chi_2$ & The scalar polynomials in $\kappa$ defined in \eqref{eq:C3formula}--\eqref{eq:defchi2} and listed together in \eqref{eq:theoremaux}. They are not functions of the variable $x$. \\
$E_1,E_2,\ T_1,T_2$ & The rational generators in \eqref{eq:DefE1E2} and the entries used to express their Frobenius action in \eqref{eq:tauE1andE2}. The lattice is $\A_LE_1\oplus V^{-1}\A_LE_2$. \\
$\mathbf u,\ \mathbf v,\ S$ & The auxiliary scalars in \eqref{eq:uv} and $S=\lambda_1^q\lambda_2+\gamma_1^q$ in the stability calculation. At a root of $\pi$, Lemma~\ref{lem:lambda1eq0} defines $S=-\chi_1^q-\gamma_1$ directly, without using $\lambda_2$. \\
$\cN(\kappa),\cH(\kappa),\cJ(\kappa)$ & The three polynomials in \eqref{eq:NHJ}; in particular $\cN=-\chi_1^q-\gamma_1^q-\gamma_1$ and $\cH=C_2(-\gamma_1)$. \\
$\pi(X),\ C_\pi,\ C_\pi^\circ$ & The polynomial in \eqref{eq:pi-expanded}, the affine parameter curve $Y^{q+1}=\pi(X)$, and its open subset $Y\ne0$. Its coordinate functions are denoted $X=\kappa$ and $Y=\lambda_1$. \\
$\bar C_\pi,\ P_\infty$ & The smooth projective completion and its unique point over $X=\infty$. The point $P_\infty$ lies on the parameter curve and is distinct in meaning from $\infty\in E$. \\
$\mathcal X_V,\ \mathcal X_V/{\simeq}$ & The set of actual sign-normalized Drinfeld homomorphisms with fixed determinant motive, and the set of their isomorphism classes. The latter is obtained by scalar conjugation. \\
$B_0,B_1,B_2$ & The coefficients in \eqref{eq:defB} of the quadratic relation used when $P_3=0$. \\
$Q,\ \hat f,\ \hat m$ & A solution of $Q\boxplus Q^{(1)}=V$, the function satisfying $\tau^2\s=\hat f\s$, and its line slope in \eqref{eq:hatf}. They describe the modules with $P_3=0$. \\
\end{longtable}
\endgroup

\section{Elimination and coefficient comparison}
We record the algebra needed for the explicit formulas. All products
below are products in $L\{\tau\}$, so Frobenius twists of coefficients
must be retained.

\subsection{Line intersections and the auxiliary identities}\label{app:auxiliaryidentities}
Put $l(x)=\beta^\vee+\kappa(x-\alpha)$.
Then $\mathfrak L_\kappa=y-l(x)$ and
$C_3=y+a_1x+a_3+l(x)$.
The Weierstrass equation gives
\begin{align*}
 \mathfrak L_\kappa C_3
 &=y^2+(a_1x+a_3)y-l(x)^2-(a_1x+a_3)l(x)\\
 &=x^3+a_2x^2+a_4x+a_6-l(x)^2-(a_1x+a_3)l(x)\\
 &=(x-\alpha)(x^2-\chi_1x+\chi_2).
\end{align*}
The factor $x-\alpha$ occurs because $(\alpha,\beta^\vee)$
lies on both the line and the curve.
The coefficient of $x^2$ in the cubic is
$a_2-\kappa^2-a_1\kappa=-\chi_1-\alpha$.
Writing $b_\kappa=\beta^\vee-\kappa\alpha$, its coefficient of $x$ is
$a_4-2\kappa b_\kappa-a_1b_\kappa-a_3\kappa=\chi_2+\alpha\chi_1$.
These two comparisons give exactly \eqref{eq:aux1} and
\eqref{eq:defchi2}. The displayed factorization also proves
$C_3=hC_2$, including tangent intersections.

We next verify the two identities \eqref{eq:unconditional}.
For the first, expansion gives
\begin{align*}
 \uk\gamma_2+\chi_1^q+\gamma_1
 &=(\kappa^q-\kappa)(-\kappa^q-m-a_1)
   +\kappa^{2q}+a_1\kappa^q-a_2-\alpha^q\\
 &\qquad+(\kappa^q+m+a_1)(m-\kappa)-\theta\\
 &=m^2+a_1m-a_2-\alpha^q-\theta=\alpha.
\end{align*}
For the second, put $b_m=\eta-m\theta$ and
$K_\kappa=\kappa^q+m+a_1=-\gamma_2$.
The line of slope $m$ gives
\[
 \beta^\vee=m\alpha+b_m,\qquad
 \beta^q=m\alpha^q+b_m,\qquad
 (\beta^\vee)^q=-(m+a_1)\alpha^q-b_m-a_3.
\]
Substitution in $\vk$ and $\gamma_1$ first gives
\[
 \gamma_2\vk-\alpha\gamma_1
 =\alpha^qK_\kappa^2+(2b_m+a_3)K_\kappa+\alpha\theta.
\]
On the other hand, subtracting this expression from the formula
for $\chi_2^q$ gives
\begin{align*}
 \chi_2^q-(\gamma_2\vk-\alpha\gamma_1)
 &=\alpha^q\bigl(a_2+\alpha^q-m(m+a_1)\bigr)\\
 &\quad+a_4-(2m+a_1)b_m-ma_3-\alpha\theta\\
 &=a_4-(2m+a_1)b_m-ma_3
       -\alpha\alpha^q-\alpha^q\theta-\alpha\theta\\
 &=0.
\end{align*}
The last equality is the $x$-coefficient comparison for the cubic
obtained by substituting $y=mx+b_m$ in the Weierstrass equation;
its roots are $\alpha,\alpha^q,\theta$.
This verifies the identity without dividing by $\gamma_2$.

\subsection{Twisted-polynomial division and elimination}\label{app:TwistedDivision}
First,
\[
 (\phi_y-\eta)
 -\bigl(\tau^2+(h_5-g_3^{q^2})\tau+P_0\bigr)(\phi_x-\theta)
 =P_1\tau+P_2\tau^2+P_3\tau^3.
\]
Indeed, the coefficients of $\tau^6,\tau^5,\tau^4$ on the left
are respectively
\[
 1-1,\qquad
 h_5-g_3^{q^2}-(h_5-g_3^{q^2}),\qquad
 h_4-g_2^{q^2}-(h_5-g_3^{q^2})g_3^q-P_0,
\]
and all are zero. The coefficients of $\tau,\tau^2,\tau^3$
are precisely $P_1,P_2,P_3$.
Applying this to $\s$ gives \eqref{eq:first}. Put
\[
\begin{aligned}
 w&=P_3^qg_3-P_2^q,\\
 R_1&=P_3^{q+1}g_1-wP_1,&
 R_2&=P_3(P_3^qg_2-P_1^q)-wP_2,\\
 L_0&=wP_0+P_3^{q+1},&
 L_1&=P_3P_0^q+w(h_5-g_3^{q^2}),\\
 L_2&=P_3(h_5-g_3^{q^2})^q+w.
\end{aligned}
\]
Subtracting the Frobenius multiple gives
\[
 P_3^{q+1}(\phi_x-\theta)
 -(P_3\tau+w)(P_1\tau+P_2\tau^2+P_3\tau^3)=R_1\tau+R_2\tau^2.
\]
After substituting the preceding division identity and applying to
$\s$, this becomes
\begin{align*}
0={}&P_3(x-\theta^{q^3})\tau^3\s+
       \bigl(L_2(x-\theta^{q^2})-R_2\bigr)\tau^2\s\\
 &+\bigl(L_1(x-\theta^q)-P_3(y-\eta^q)-R_1\bigr)\tau\s\\
 &+\bigl(L_0(x-\theta)-w(y-\eta)\bigr)\s .
\end{align*}
Subtract this equation from $(x-\theta^{q^3})$ times \eqref{eq:first}.
The result is \eqref{eq:motiverelation}, with
\begin{align}
 C_2={}&(x-\theta^{q^3})(x-\theta^{q^2}+P_2)
                  -L_2(x-\theta^{q^2})+R_2,\label{eq:C2formula}\\
 C_1={}&-(x-\theta^{q^3})\bigl(P_1+(h_5-g_3^{q^2})(x-\theta^q)\bigr)
           \notag\\*
       &+L_1(x-\theta^q)-P_3(y-\eta^q)-R_1,\label{eq:C1formula}\\
 C_0={}&-(x-\theta^{q^3})\bigl(P_0(x-\theta)-(y-\eta)\bigr)
                       +L_0(x-\theta)-w(y-\eta).\label{eq:C0formula}
\end{align}
These formulas prove the coefficient assertions in
Proposition~\ref{prop:motive2}.

\subsection{Recovering every Drinfeld coefficient}\label{app:coeff}
For a comparison with \eqref{eq:geometricC}, the coefficients of
$x^2,y,x^2,y$ in $C_1,C_1,C_0,C_0$ respectively give
\begin{equation}\label{eq:comparison1}
 \lambda_2=-h_5+g_3^{q^2},\quad
 \lambda_1=-P_3,\quad
 P_0=-(\kappa+a_1+\gamma_2)=m+\kappa^q-\kappa,\quad
 w=-\gamma_1-\theta^{q^3}.
\end{equation}
It follows that
\begin{align}
 L_2&=\lambda_1\lambda_2^q-\gamma_1-\theta^{q^3},\notag\\
 P_2&=\lambda_1\lambda_2^q-\gamma_1-\chi_1+\theta^{q^2},\label{eq:P2explicit}\\
 L_1&=-\lambda_1P_0^q+(\gamma_1+\theta^{q^3})\lambda_2,\notag\\
 P_1&=-\lambda_1(P_0^q+\kappa+a_1)
                  -\lambda_2(\theta^q-\gamma_1-\chi_1).\label{eq:P1explicit}
\end{align}
The first and third are the definitions of $L_2,L_1$ after
\eqref{eq:comparison1}. Comparing the $x$ coefficients of $C_2$ and
$C_1$ gives the other two.

The constant coefficients of \eqref{eq:C2formula} and
\eqref{eq:C1formula} yield
\begin{align*}
 R_2={}&\chi_2-\theta^{q^3}\theta^{q^2}
                    +P_2\theta^{q^3}-L_2\theta^{q^2},\\
 R_1={}&\lambda_2(\theta^{q^3}\theta^q-\chi_2)
          +P_1\theta^{q^3}-L_1\theta^q-\lambda_1\eta^q\\
       &+\lambda_1\bigl(\beta+(\kappa+a_1)\alpha\bigr).
\end{align*}
We now recover $g_1,g_2,g_3$.
In this calculation only, abbreviate
\[
 F=P_0^q+\kappa+a_1,\qquad
 B=\beta+(\kappa+a_1)\alpha.
\]
The equation $w=P_3^qg_3-P_2^q$, together with
$P_3=-\lambda_1$, gives
\begin{align*}
 -\lambda_1^qg_3
 &=w+P_2^q\\
 &=-\gamma_1-\theta^{q^3}
   +\lambda_1^q\lambda_2^{q^2}
   -\gamma_1^q-\chi_1^q+\theta^{q^3}\\
 &=\lambda_1^q\lambda_2^{q^2}+\cN.
\end{align*}
Since $\cN=\lambda_1^q\lambda_2$, division by $\lambda_1^q$
gives $g_3=-\lambda_2^{q^2}-\lambda_2$.

For $g_2$, the definition of $R_2$ reads
\[
 R_2=\lambda_1^{q+1}g_2+\lambda_1P_1^q
           +(\gamma_1+\theta^{q^3})P_2.
\]
Equating this with its constant-coefficient expression above yields
\begin{align*}
 \lambda_1^{q+1}g_2
 &=\chi_2-\theta^{q^3}\theta^{q^2}
       -L_2\theta^{q^2}-\gamma_1P_2-\lambda_1P_1^q\\
 &=\cH-\lambda_1\lambda_2^q(\theta^{q^2}+\gamma_1)
       -\lambda_1P_1^q\\
 &=\cH+\lambda_1^{q+1}F^q
       -\lambda_1\lambda_2^q(\gamma_1+\gamma_1^q+\chi_1^q)\\
 &=\cH+\lambda_1^{q+1}F^q+\lambda_1\lambda_2^q\cN.
\end{align*}
In the second line we used the displayed formulas for $L_2,P_2$
and $\cH=\chi_2+\gamma_1(\gamma_1+\chi_1)$.
The third uses
$P_1^q=-\lambda_1^qF^q
-\lambda_2^q(\theta^{q^2}-\gamma_1^q-\chi_1^q)$.
Dividing by $\lambda_1^{q+1}=\pi(\kappa)$ now gives the formula
for $g_2$.

For $g_1$, the definition of $R_1$ becomes
\[
 R_1=\lambda_1^{q+1}g_1+(\gamma_1+\theta^{q^3})P_1.
\]
Its constant-coefficient expression therefore gives
\begin{align*}
 \lambda_1^{q+1}g_1
 &=\lambda_2(\theta^{q^3}\theta^q-\chi_2)
       -\gamma_1P_1-L_1\theta^q+\lambda_1(B-\eta^q)\\
 &=-\lambda_2(\chi_2+\gamma_1\theta^q)
       -\gamma_1P_1+\lambda_1P_0^q\theta^q
       +\lambda_1(B-\eta^q)\\
 &=-\lambda_2\cH+
       \lambda_1\bigl(\gamma_1F+P_0^q\theta^q+B-\eta^q\bigr)\\
 &=-\lambda_2\cH+\lambda_1J.
\end{align*}
The second and third lines use $L_1$ and $P_1$, respectively.
The last uses $P_0^q=m^q+\kappa^{q^2}-\kappa^q$ and the definition
of $\cJ$. Division by $\pi(\kappa)$ gives
$g_1=-\lambda_2\cH/\pi+\cJ/\lambda_1^q$.

Finally the remaining coefficients are recovered in a triangular
order:
\begin{align*}
 h_5&=g_3^{q^2}-\lambda_2,\\
 h_4&=P_0+g_2^{q^2}-\lambda_2g_3^q,\\
 h_3&=P_3+g_1^{q^2}-\lambda_2g_2^q+P_0g_3,\\
 h_2&=P_2-\lambda_2g_1^q+P_0g_2,\\
 h_1&=P_1+P_0g_1.
\end{align*}
Substituting $P_3=-\lambda_1$ and
\eqref{eq:P2explicit}--\eqref{eq:P1explicit} gives all five
formulas in \eqref{eq:allcoefficients}.
Thus every coefficient of the module constructed in
Proposition~\ref{prop:converse} is determined by the stated parameters.

\bibliographystyle{amsplain}
\bibliography{paper}

\end{document}